\documentclass[a4paper,10pt, twoside]{article}
\usepackage{bbm}

\usepackage[sort&compress, square, numbers]{natbib}

\usepackage[a4paper]{geometry} 
\usepackage[utf8]{inputenc}
\usepackage[T1]{fontenc}
\usepackage[english]{babel}
\usepackage{lmodern}
\usepackage{amssymb,amsthm,amscd}
\usepackage[fleqn]{amsmath}
\usepackage{extarrows}
\usepackage[pdfborder={0 0 0},colorlinks,citecolor=blue,urlcolor=blue]{hyperref}
\usepackage{mathrsfs}
\usepackage{appendix}
\usepackage[usenames]{color}

\theoremstyle{plain}
\newtheorem{theorem}{Theorem}[section]
\newtheorem{corollary}[theorem]{Corollary}

\newtheorem{lemma}[theorem]{Lemma}

\theoremstyle{definition}

\theoremstyle{remark}
\newtheorem{remark}[theorem]{Remark}

\numberwithin{equation}{section}

\def\titlerunning#1{\gdef\titrun{#1}}
\makeatletter
\def\author#1{\gdef\autrun{\def\and{\unskip, }#1}\gdef\@author{#1}}
\def\address#1{{\def\and{\\\hspace*{18pt}}\renewcommand{\thefootnote}{}%
		\footnote {#1}}%
	\markboth{\autrun}{\titrun}}
\makeatother
\def\email#1{E-mail: #1}
\def\subjclass#1{{\renewcommand{\thefootnote}{}%
		\footnote{\emph{Mathematics Subject Classification (2010):} #1}}}
\def\keywords#1{\par\medskip
	\noindent\textbf{Keywords.} #1}

\newcommand{\N}{\mathbb{N}}
\newcommand{\bZ}{\mathbb{Z}}
\newcommand{\R}{\mathbb{R}}

\newcommand{\calV}{\mathcal{V}}

\renewcommand{\P}{\mathbb{P}}
\newcommand{\A}{\mathbb{A}}
\newcommand{\E}{\mathbb{E}}
\newcommand{\T}{\mathbb{T}}
\newcommand{\Z}{\mathbb{Z}}

\newcommand{\D}{\mathscr{D}}

\newcommand{\I}{ {\textsc{I}}}
\newcommand{\J}{ \textsc{J} }

\newcommand{\abs}[1]{\left\lvert#1\right\rvert}
\newcommand{\ind}[1]{{\mathbf{1}_{\left\{#1\right\}}}}

\newcommand{\norme}[1]{{\left\Vert #1 \right\Vert}}

\newcommand{\calA}{\mathcal{A}}
\newcommand{\calB}{\mathcal{B}}

\newcommand{\calH}{\mathcal{H}}

\newcommand{\rmd}{\mathrm{d}}
\newcommand{\inp}[2]{\left\langle {#1},{#2}\right\rangle}

\newcommand{\tr}[1]{\mathrm{tr}\!\left(#1\right)}
\newcommand{\e}{\mathbf{e}}

\newcommand{\norm}[1]{{\left\Vert #1 \right\Vert}}

\renewcommand{\bar}[1]{\overline{#1}}

\renewcommand{\rho}{\varrho}
\renewcommand{\epsilon}{\varepsilon}
\renewcommand{\phi}{\varphi}

\renewcommand{\tilde}[1]{\widetilde{#1}}

\newcommand{\calM}{\mathcal{M}}

\newcommand{\rmM}{\mathrm{M}}
\newcommand{\rmQ}{\mathrm{Q}}

 \allowdisplaybreaks

\begin{document}

\baselineskip=15pt


\titlerunning{A second order  expansion   in LLT for BRW}

\title{ A second order expansion in
	the local limit theorem for a branching system of
	 symmetric irreducible random walks\thanks{Supported
		by the National Natural Science Foundation of China (NSFC, Grants No. 11971063) }}
\author{Zhi-Qiang Gao}

\date{\today}

\maketitle

\address{Z.-Q. Gao: Laboratory of Mathematics and Complex Systems (Ministry of Education), School of Mathematical Sciences, Beijing Normal University, Beijing 100875, P. R. China; \email{gaozq@bnu.edu.cn}
}


\textcolor{blue}{}\global\long\def\TDD#1{{\color{red}To\, Do(#1)}}
\begin{abstract}
Consider  a    branching random walk,  where the branching mechanism is governed by a Galton-Watson process,  and the migration by a  finite range symmetric irreducible  random walk on the integer lattice $\mathbb{Z}^d$. 
Let $Z_n(z)$ be the number of the particles in the $n$-th generation at the point $z\in \Z^d$.  
Under the mild moment conditions for offspring distribution of the underlying Galton-Watson,  we derive a second order expansion in the local limit theorem  for   $Z_n(z)$ for each given $z\in \Z^d$.  That generalises the results for simple branching random walks obtained by Gao [2018, SPA].

\end{abstract}


\section{Introduction and Main result}
We consider a discrete-time branching random walk.  That describes the evolution of a population of
particles where spatial motion is present, hence generalizes the classical Galton-Watson branching processes. It has been a very active topic in probability,  because of its own importance and close connection with many other random models, e.g. multiplicative cascades, infinite particle systems, random fractals and discrete Gaussian free field. 

Since  Harris \cite [{Chapter III} \S16]{Harris63BP}  first proposed his conjecture on the question of central limit theorems for a branching random walk,   the topic has been widely studied and in various forms.  See e.g. \cite{Stam66,JoffeMoncayo73AM,AsmussenKaplan76BRW1,
	AsmussenKaplan76BRW2,Biggins90SPA,Jagers74JAP,Yoshida08,Nakashima11,
	GLW14,GL14,GK15,GL15,Kabluchko12,ChenHe2017,LouidorPerkins2015EJP} . 

  R\'ev\'esz  \cite{Revesz94} initiated the study of the convergence speed in local limit theorems for a simple branching random walk, where displacements of the particles are governed by  a simple random walk on $\bZ^d$.  Specially, R\'ev\'esz  gave a conjecture on the exact rate of the convergence speed for the simple branching random walk,  which was proved by Chen \cite{Chen2001}.  Then Gao \cite{Gao2016}  improved and modified  Chen's result. As a further generalization,  Gao \cite{Gao2018SPA}   obtained the second order asymptotic expansion of the local limit theorem  for the simple branching random walk.
      In this article, we continue the research line in \cite{Gao2018SPA} 
   by  extending the result therein to
the case  where  the migration mechanism is governed by a finite range symmetric irreducible random walk on $\Z^d$.   

We mention that some central limit theorems of different nature on the intrinsic  martingale (Biggins martingale) in  branching random walks have been established by  Gr\"ubel and Kabluchko \cite{GrubelKabluchko2016AAP}, Iksanov and Kabluchko \cite{iksanov_kabluchko_2016}, and some large deviations type results for branching random walks have been discussed by \cite{WangHuang2017JTP,HuangWangWang2020}.  For other aspects on  branching random walks, see for example, \cite{ABMY2000,GunKonigS2013EJP,HuangLiangLiu14,ILiuLiang2019,LiangLiu2020,WangHuang2019ECP,WangLiuLiuLi2019,Shi2015,Zeitouni2012}  and references therein. 

To introduce our main result, we need to make some basic settings.

In the $d$-dimensional lattice $ \mathbb{Z}^d$,  write $ \mathbf{0}=
(0,0,\cdots,0), \mathbf{1}=(1,1,\ldots, 1) $,    and let
\begin{equation*}
\mathbf{e}_1 =(1,0,\ldots, 0), \ldots, \mathbf{e}_d=(0,0, \ldots, 1) 
\end{equation*} 
be the standard  basis of  orthogonal   unit vectors.   
By convention, denote by $\inp{\cdot}{\cdot}$ the inner product in $\R^d$ and $\|\cdot\|$ the norm therein, i.e.
for $x=(x_1,\cdots,x_d)\in\mathbb{R}^d$ and $y=(y_1,\cdots,y_d)\in\mathbb{R}^d$, 
\begin{equation*}
\langle{x,y}\rangle=x_1y_1+x_2y_2+\cdots+x_dy_d,\quad \|x\|=\sqrt{x_1^2+x_2^2+\cdots+x_d^2}.
\end{equation*}
The multiplication of a vector $x\in \R^d$     by a real number $ \lambda \in \R$      is defined by
$
\lambda {x}=(\lambda x_1, \lambda x_2,\cdots, \lambda x_d).
$

Let $\N=\{0,1,2,\cdots\} $.  Suppose that $N$  is an integer-valued random variable  with  the probability law
\begin{equation*}
\P(N=k)=p_k,  \quad {k\in \N},  \quad    ~~\sum_{k=0}^\infty p_k=1,   
\end{equation*}
and 
   that $L$ is   a random variable with the probability law   
\begin{equation}\label{SRW}
\P(L= \mathbf{0})=\zeta_0, \quad  \P(L=r  \mathbf{e}_s)=\P(L=  -r\mathbf{e}_s)=\frac{1}{2} \zeta_{s,r},\quad 1\leq s\leq d, 1\leq r \leq t_s,
\end{equation}
where    each $t_s$  ($1\leq s\leq d$)    is a positive integer,   $ \zeta_0\in [0,1)$, $\zeta_{s,r} \in [0,1) $,      and 
 \begin{equation*}
 \zeta_{s,t_s}>0,   \quad  \quad \zeta_0+\sum_{s=1}^d \sum_{r=1}^{t_s}\zeta_{s,r}= 1. 
 \end{equation*}
For $ k\in \N$, set 
\begin{equation*}
\zeta_s(k) =\sum_{r=1}^{t_s} \zeta_{s,r} r^k,\quad  1\leq s\leq d.
\end{equation*}
Denote by $ \Gamma_k= \mathrm{diag}\big(\zeta_1(k), \cdots, \zeta_d(k)\big)$ the diagonal matrix with diagonal entries  $\zeta_1(k), \cdots, \zeta_d(k)$.  It is easy to see that 
\begin{equation*}
\det \Gamma_k= \prod_{s=1}^{d}\zeta_s(k),\quad \mathrm{tr} (\Gamma_k^{-1})= \sum_{r=1}^d \big(\zeta_r (k)\big)^{-1},
\end{equation*}   
where   $\det M $  is the  determinant of a $d\times d$ matrix $M$, $M^{-1}$ is  its inverse,  and $\mathrm{tr} (M)$ is its  trace.
By convention,  for a vector $x$, the notion  $M x$  is viewed as  the product of $M$ and  the  column matrix $x$.

Throughout the article,   we shall  assume that the law \eqref{SRW} of $L$ satisfies
\begin{equation}\label{irr}
\gcd \{r : \zeta_{r,s}>0 \}=1, \quad s=1,2,\cdots,d, 
\end{equation}
where $\gcd$ denotes the greatest common divisor.
This assumption implies that 
 $V=\{r\mathbf{e}_s :  \zeta_{s,r} >0, 1\leq r\leq t_s, s=1,2,\cdots,d\}$ is a  \emph{generating set}  of $\Z^d$, which means that  
\begin{equation}\label{ir}
\forall  y\in \Z^d,  \  \ \exists  ~ \{ k_{r,s} \}  \subset \Z ,  \quad \mbox{s.t. }
y= \sum_{r\mathbf{e}_s\in V}  k_{r,s} r\mathbf{e}_s. 
\end{equation}
Moreover,   the random walk $S_n$ with such increment  distribution  $L$ must be    \emph{irreducible} (meaning that each point in $\Z^d$ can be reached  with positive probability \cite{Lawler2010}).

Under the hypothesis \eqref{irr}   for $L$, 
there  are two possible cases:
\begin{itemize}
	\item[(Ha)]  there exists one $s$ such that 
	 either the set  $\{ r: \zeta_{s,r} >0\}  $ contains  at least one odd  integer and one even, 
	 or  the set  $ \{ r: \zeta_{s,r} >0\}  $ only contains odd integers  and $\zeta_0>0$;
   \item[(Hb)] $\zeta_0=0$ and for each $1\leq s\leq d$,  $ \{ r: \zeta_{s,r} >0\}  $ only contains odd integers.
\end{itemize}

Denote by $\calA_d$ the set of all probability distributions $L$  with the law $ \eqref{SRW}$    satisfying  \eqref{irr} and (Ha). When $L\in \calA_d$,   the random walk with  increment  distribution  $L$   is  \emph{aperiodic}, which means that  each point on $\mathbb{Z}^d$ can  be reached   after $n$ steps with positive probability for all $n$ sufficiently  large.  

 
 Denote by $\calB_d$ the set of all probability distributions $L$  with the law $ \eqref{SRW}$ satisfying \eqref{irr} and (Hb). For $L\in \calB_d$,   the associated random walk is \emph{bipartite}, that means  the random walks starting from a given point $x_0$ return to  $x_0$ only after an even number of steps.   In this case, $\Z^d$ is divided into two disjoint sets $\Z_o$ and $\Z_e$, such that  the walk starting form the origin reaches  the states set $\mathbb{Z}_o$ in an  odd number of steps and reaches  $\mathbb{Z}_e$ in an even number of  steps.

  For example, if     $L$ with  the law  \eqref{SRW}  satisfies
 \begin{equation*}
 \zeta_0=\sigma\in[0,1), \qquad \zeta_{s,1}=(1-\sigma)/d,\quad  1\leq s \leq d,
 \end{equation*}  
 then  when $\sigma>0$,   $L\in \calA_d$ and the associated random walk is
 aperiodic;  whereas  when $\sigma=0$,  $L \in\calB_d$ 	and the associated random walk is bipartite .

Consider a  discrete-time branching random walk in the $d$-dimensional integer lattice $\mathbb{Z}^d$. 
Initially, an ancestor particle $\varnothing$ is located at $S_\varnothing=\mathbf{0} \in\mathbb{Z}^d$. 
At time $1$,   it is replaced by $N=N_\varnothing$ particles ($\varnothing 1, \cdots,\varnothing N $), and each particle $\varnothing i( 1\leq i\leq N)$   moves to  $S_{\varnothing i}=S_{\varnothing}+ L_{\varnothing i}$.        In general, at time $n+1$, each particle of generation $n$,  which is denoted by a sequence of positive integers of length $n$, i.e.  $u=u_1u_2\cdots u_n $,  
is replaced   by   $ N_u$ new particles of generation $n+1$,  $\{ ui:1\leq i\leq N_u \}$, with  displacements $ L_{u1}, L_{u2}, \cdots, L_{uN_u}$. That means  for $1\leq i\leq N_u$,  each particle $u i $ moves to $ S_{ui}= S_u+L_{ui} $.   Here, all $N_u$ and $L_u$,  indexed by finite sequences of integers $u$, are independent random elements defined  on some probability space $(\Omega,\mathcal{F}, \mathbb{P})$: all $N_u$ are independent copies of the integer-valued random variable $N$;  
all  $L_u$  are  independent random vectors  identically distributed (abbreviated as i.i.d) with $L$ defined by the law \eqref{SRW}.

The genealogy of all particles  forms a Galton-Watson tree $\T$ with $\{N_u\}$ as defining elements. Let  
\begin{equation*}
\mathbb{T}_n =\{u\in
\mathbb{T} :|u|=n\} 
\end{equation*} 
be the set of particles of generation $n$, where $|u|$ denotes the length of the
sequence $u$ and represents the number of generation to which $u$ belongs. 

Let $Z_n(\cdot)$ be the counting measure of particles in the $n$-th generation: for $B\subset \Z^d$,
\begin{equation*}
Z_n(B)= \sum_{u\in \T_n}\mathbf{1}_{B}(S_u).
\end{equation*}  
In particular, we will frequently write $Z_n(z)$  instead of $Z_n(\{z\})$, which  is the number of the $n$-th generation individuals located at $z\in \bZ^d$ by definition.


    By the definitions,  $\{Z_n(\Z^d)\}_n$ forms a  Galton-Watson process.  Throughout the paper, we  assume
    \begin{equation}\label{SC}
    m=\E N=\sum_{k=1}^\infty kp_k >1,
    \end{equation} 
   and the process  is  called \textit{supercritical}. 
 In this case, the process  survives with  strictly  positive probability.
 For ease of notion, we always  assume that  
 \begin{equation}\label{NE}
 \P(N=0)=0,
 \end{equation}  
 that implies that  
 the  process \textit{survives  with probability 1}. 
 Obviously this  assumption could be removed, but then all results hold conditionally on
 non-extinction.
Under this setting, it holds that   $\P (Z_n(\bZ^d) \rightarrow \infty) =1.$ 
 Provided that $\E N\log N<\infty$,  the   Kesten-Stigum
 theorem (\cite{AthreyaNey72,AsmussenHering1983})  asserts that the sequence $W_n= Z_n(\bZ^d)/m^n$ converges almost surely (abbreviated as  a.s.) to a finite positive random variable $  W$. 
 


With the above notion, our main result can be stated as follows.
\begin{theorem}	\label{SBRW-Th} Assume  that the conditions \eqref{SC} and\eqref{NE} hold,   $ \E N(\ln N)^{1+\lambda}<\infty$ for some  $\lambda>3(d+6)$  and $L$  obeys the law   \eqref{SRW}.   
	Then there exist some random variables  $\mathcal{V}_1, \mathcal{V}_2 ,\mathcal{V}_3  \in \R ^d$   and $\mathcal{V}^z_{2} ,\mathcal{V}_4\in \R$, such that   for each $z=(z_1,z_2,\ldots, z_d)\in \bZ^d$, as   $n\rightarrow \infty$,  
\\
 	 (I)  in the case $L\in \calA_d $,
 	 \begin{equation}
 	 \frac{1}{m^n} Z_n(z)= \dfrac{(2\pi n)^{-d/2}}{ \sqrt{ \det \Gamma_2} }    
 	 \bigg[   W + \frac{1}{n}  F_1(z)+\frac{1}{n^2}  F_2(z)  \bigg]  +  \frac{1}{n^{2+d/2}}o(1) \quad    a.s.,
 	 \end{equation} 
	 (II) in the case $L\in \calB_d$,   provided that $ n\equiv z_1+z_2+\cdots+z_d\, (\,\mathrm{ mod }~  2)$, 
	 \begin{equation}
	 \frac{1}{m^n} Z_n(z)= \dfrac{2(2\pi n)^{-d/2}}{ \sqrt{ \det \Gamma_2} }    
	 \bigg[   W + \frac{1}{n}  F_1(z)+\frac{1}{n^2}  F_2(z)  \bigg]  +  \frac{1}{n^{2+d/2}}o(1)   \quad a.s., 
	 \end{equation}
 where \begin{align}
 \label{f1}	  F_1(z)&=   \Big(\tau_d -\frac{1}{2} \inp{z}{\Gamma_2^{-1} z} \Big)W+ \inp{\mathcal{V}_1}{\Gamma_2^{-1} z}-  \frac{1}{2} \inp{\mathcal{V}_2 }{ \Gamma_2 ^{-1}\mathbf{1}},\\ 
 \label{eq-taud}\tau_d&=\dfrac{1}{8} \mathrm{tr} (\Gamma_4\Gamma_2^{-2}) -\dfrac{1}{8} d(d+2), 
 \\ 
 \nonumber 	  F_2(z)&= \Big(\dfrac{1}{8} \inp{\Gamma_2^{-1}z}{z}^2    -  \inp{\Lambda_d z}{z}+\chi_d\Big)W+  \inp{\calV_1}{\Big(2\Lambda_d-\frac{1}{2} \inp{z}{ \Gamma_2^{-1}z}\Gamma_2^{-1}  \Big)z}  
 \\
 \label{f2}
 &  \quad+ \inp{\calV_2}{\Big(\frac{1}{4} \inp{z}{\Gamma_2^{-1}z}\Gamma_2^{-1}-\Lambda_d\Big)\mathbf{1}}  + \frac{1}{2} \calV_2^{z}-\frac{1}{2}\inp{ \calV_3}{\Gamma_2^{-1}z}+\frac{1}{8}\calV_4,
 \\
 \label{eq-Lambdad}\Lambda_d&=\dfrac{1}{16}\Big(\tr{\Gamma_4\Gamma_2^{-2}}-(d+2)(d+4)\Big)\Gamma_2^{-1}  +\frac{1}{4} \Gamma_4 \Gamma_2^{-3},
 \\
 \nonumber  \chi_d&=   -\dfrac{1}{64}(d+2)(d+4)\tr{\Gamma_4\Gamma_2^{-2}} + \frac{1}{12}\tr{\Gamma_4^2\Gamma_2^{-4}}+
 \frac{1}{128}\Big(\tr{\Gamma_4\Gamma_2^{-2}}\Big)^2
 \\  \label{eq-Cd}
 &\qquad\qquad\qquad\qquad\qquad\qquad\qquad  -\frac{1}{48}\tr{\Gamma_6\Gamma_2^{-3}}
 +\dfrac{1}{384}d(d+2)(d+4) (3d+2).  
 \end{align} 
\end{theorem}

\begin{remark}\label{Rmk1}
In Section \ref{Sec3}, we shall give the detailed accounts on the quantities $\mathcal{V}_1,  \mathcal{V}_2 ,\mathcal{V}_3  \in \R ^d$
and  	$ \calV_2^z, \mathcal{V}_4\in \R$.  We prove that these random variables are limits of the following  sequences:    
 \begin{equation*}
  \mathcal{V}_j  \xlongequal{{a.s.}}  \lim_{n\rightarrow \infty} N_{ j,n} \quad  j=1,2,3,4;   \quad \calV_2^z \xlongequal{{a.s.}}  \lim_{n\rightarrow \infty} N_{ 2,n}^z,
 \end{equation*}
   where
	\begin{align*}
	&N_{1,n}= \frac{1}{m^n}\sum_{u\in \T_n} S_u,\\
	&N_{2,n}= \frac{1}{m^n}\sum_{u\in \T_n}\Big( \inp{S_u}{\e_1}^2-n\zeta_1(2), \inp{S_u}{\e_2}^2-n\zeta_2(2),\cdots, \inp{S_u}{\e_d}^2-n\zeta_d(2)\Big),	
\\ & N_{2,n}^z=   \frac{1}{m^n}\sum_{u\in \T_n}\left[\inp{S_u}{\Gamma_2^{-1}z}^2  - n\inp{\Gamma_2^{-1}z }{z}\right], 
	\\ & N_{3,n}=  \frac{1}{m^n}\sum_{u\in \T_n} \left[  \inp{S_u}{\Gamma_2^{-1}S_u} S_u    -(d+2)n S_u\right],
	\\ &  N_{4,n}=   \frac{1}{m^n}\sum_{u\in \T_n}  \left[  \inp{S_u}{\Gamma_2^{-1}S_u}^2- (4+2d) n  \inp{S_u}{\Gamma_2^{-1}S_u} +d(d+2) (n^2+  n) - \mathrm{tr} (\Gamma_4 \Gamma_2^{-2}) n \right].
	\end{align*}
	Moreover we also consider the convergence rate of the above limit procedure in the next section. Those  play important roles in the proof of  Theorem \ref{SBRW-Th}.
\end{remark}


  Comparing the bipartite case with the aperiodic case,   a phase transition occurs. 	The result indicates  that the   periodicity of the random walk alternates the formulation of the limit laws with a factor $2$. 
	 In the following Corollary \ref{SBRW-Cor}, we discuss the example mentioned before  when the underlying  moving law is governed by  a  simple random walk or its lazy version (meaning stay there  with positive probability).

\begin{corollary}\label{SBRW-Cor}
Suppose that	 the conditions \eqref{SC} and \eqref{NE} hold,    $ \E N(\ln N)^{1+\lambda}<\infty$ for some  $\lambda>3(d+6)$,  and  the law of $L$ in \eqref{SRW} satisfies   $\zeta_0= \sigma   ,\zeta_{1,1}=\zeta_{2,1}=\cdots= \zeta_{d,1}=(1-\sigma )/d. $
	Then there exist some random variables $ \mathcal{V}_1,\mathcal{V}_2 ,   \widetilde{\calV_2^z},\widetilde{ \calV_3},\widetilde{ \calV_4}   $, such that   for each $z=(z_1,z_2,\ldots, z_d)\in \bZ^d$, as   $n\rightarrow \infty$, 
	\\ 
	(\textbf{I}) when $\sigma>0$,   
	\begin{equation*}
	\frac{1}{m^n} Z_n(z) =   \Big(\dfrac{d }{ 2\pi n(1-\sigma)  }\Big)^{d/2}    
	\bigg[   W + \frac{1}{n}   {\calH_{\sigma,1}} (z)+\frac{1}{n^2}  {\calH_{\sigma,2}} (z) \bigg]  +  \frac{1}{n^{2+d/2}}o(1)\quad a.s., 
	\end{equation*}
		(\textbf{II}) when  $\sigma=0$, 
		provided that $ n\equiv z_1+z_2+\cdots+z_d\, (\,\mathrm{ mod }~  2)$, 
		\begin{equation*}
		\frac{1}{m^n} Z_n(z) =  2 \Big(\dfrac{d }{ 2\pi n  }\Big)^{d/2}    
		\bigg[   W + \frac{1}{n} { \calH_{0,1}}(z)+\frac{1}{n^2}  {\calH_{0,2}}(z)  \bigg]  +  \frac{1}{n^{2+d/2}}o(1)\quad a.s., 
		\end{equation*}
	where  for $\sigma\in [0,1)$,
	\begin{align*}
	 {\calH_{\sigma,1}} (z)&=\frac{d}{1-\sigma} \bigg[ \Big(\frac{1}{8}\sigma(d+2) - \frac{1}{4}-\frac{1}{2}\norm{z}^2\Big)W+ \inp{z}{\calV_1}-\frac{1}{2}
	\inp{\calV_2}{\mathbf{1}}\bigg],
	\\
	  {\calH_{\sigma,2}} (z)&= \frac{d^2 }{ (1-\sigma)^2}\bigg[ \Big( \frac{1}{8}\norm{z}^4+\mu_{\sigma,d} \norm{z}^2+\chi_{\sigma,d}\Big)W+\Big(2\mu_{\sigma,d}-\frac{1}{2}\norm{z}^2\Big)\inp{\calV_1}{z}
	\\ 
	&\qquad\qquad\qquad\qquad+ \Big(\frac{\norm{z}^2}{4}-\mu_{\sigma,d}\Big)\inp{ \calV_2}{\mathbf{1}}+\frac{1}{2}\widetilde{ \calV_2^z}-\frac{1}{2}\inp{z}{\widetilde{ \calV_3}} +\frac{1}{8}\widetilde{\calV_4}\bigg],
\\ \mbox{ with~~~ }& 
	\mu_{\sigma,d}= - \frac{1}{8}\Big(1+\frac{4}{d}\Big) +\sigma\Big(\frac{d}{16}+\frac{3}{8}+\frac{1}{2d}\Big),
 \\ & 
	\chi_{\sigma,d}= \frac{d}{48}-\frac{1}{32}+\frac{1}{24d}+\sigma\frac{1}{64}(d+2)(d+4)\Big( \frac{\sigma}{2}+ \frac{\sigma-2}{3d}\Big ).
 	\end{align*} 
 

\end{corollary}
\begin{remark}
In this corollary,	the random variables $ \mathcal{V}_1,\mathcal{V}_2 ,   \widetilde{\calV_2^z},\widetilde{ \calV_3},\widetilde{ \calV_4}   $ are limits of the following sequences:
	\begin{equation*}
	\mathcal{V}_j  \xlongequal{{a.s.}}  \lim_{n\rightarrow \infty} N_{ j,n} ,\quad  j=1,2, ;  ~~ \widetilde{ \calV_j} \xlongequal{{a.s.}}  \lim_{n\rightarrow \infty} \widetilde{N_{ j,n}}, \quad  j=3,4, 
	\quad \widetilde{\calV_2^z} \xlongequal{{a.s.}}  \lim_{n\rightarrow \infty} \widetilde{N_{ 2,n}^z},
	\end{equation*}
	with
	\begin{align*}
	&N_{1,n}= \frac{1}{m^n}\sum_{u\in \T_n} S_u,\\
	&N_{2,n}= \frac{1}{m^n}\sum_{u\in \T_n}\Big( \inp{S_u}{\e_1}^2-\frac{n(1-\sigma)}{d} , \inp{S_u}{\e_2}^2-\frac{n(1-\sigma)}{d} ,\cdots, \inp{S_u}{\e_d}^2-\frac{n(1-\sigma)}{d} \Big),	
	\\ & \widetilde{N_{2,n}^z}=   \frac{1}{m^n}\sum_{u\in \T_n}\left(\inp{S_u}{z}^2  - \frac{n(1-\sigma)}{d}\norm{z}^2\right), 
	\\ & \widetilde{N_{3,n}}=  \frac{1}{m^n}\sum_{u\in \T_n} \left[  \norm{S_u}^2 S_u    -\Big(1+\frac{2}{d}\Big)n(1-\sigma) S_u\right],
	\\ & \widetilde{N_{4,n}} =   \frac{1}{m^n}\sum_{u\in \T_n}  \left[   \norm{S_u}^4- \Big(2+\frac{4}{d}\Big) n (1-\sigma) \norm{S_u}^2+ \Big(1+\frac{2}{d}\Big) \Big(
	(1-\sigma)^2 n^2-\sigma n \Big)+ \frac{2}{d} n \right].
	\end{align*}	 
\end{remark}

\begin{remark}
	 Corollary  \ref{SBRW-Cor}  generalises the main results of \cite{Gao2016,Gao2018SPA}. Moreover,
 the expansions in  Corollary  \ref{SBRW-Cor} (II)    was obtained in \cite[Theorem 1.1]{Gao2018SPA} under the stronger moment condition $\lambda>6(d+5)$.
\end{remark}

At the end of this section, we  comment briefly on the strategy of proofs. Although the starting general ideas herein  are  close to those of  \cite{Gao2016,Gao2018SPA},  the details of the proofs require more refined estimates and the arguments  are tricky. Specially what sets this work aside is   the explicit expressions of the expansion terms.    On the technical level,  the proofs need  much more complicated   calculations.

The rest of the paper is organized as follows.  In the next section, we discuss  the convergence properties of the sequences $\{N_{1,n}\}, \{N_{2,n}\}, \{N_{3,n}\}$, $\{N_{4,n}\}$,  and $\{N_{2,n}^z\}$, as mentioned in Remark  \ref{Rmk1}.
In Section \ref{Sec2}, we derive the second order expansion  in the local limit theorem for a symmetric 
random walk on the $d$-dimensional integer lattice, which is used in the proof of Theorem \ref{SBRW-Th}.  With the help of those  preliminary results, the last section  is
devoted to the proof of Theorem \ref{SBRW-Th}.

\section{Preliminary results}\label{Sec3}
In this section, we shall describe the   $\mathcal{V}_1,  \mathcal{V}_2 ,\mathcal{V}_3  \in \R ^d$ and  	$ \calV_2^z, \mathcal{V}_4\in \R$ appeared in the main theorem. 
We show that they are limits of some martingale sequences  mentioned in Remark \ref{Rmk1}.

Set
\begin{equation}\label{D-fil}
\mathscr{D}_0=\{\emptyset,\Omega \}, \quad  \mathscr{D}_n  =  \sigma (  N_u, L_{ui}:  i\geq 1, |u| <
n)   \mbox{ for  $n\geq 1$}.
\end{equation}
Denote respectively by $\P_{\D_n}$ and $\E_{\D_n}$  the conditional probability  and conditional expectation with respect to $\D_n$. 

The results, which we shall need in the proof of Theorem \ref{SBRW-Th}, are sated as follows: 
\begin{theorem}\label{thmCRM} 
	The sequences $ \{ N_{q,n} \} (q=1,2,3,4)$ and $\{N_{2,n}^z\}$ are martingales  w.r.t. the filtration $\{\D_{n}\}_{n\geq0} $. 
Moreover, if \eqref{SC} and \eqref{NE} hold, $\E N(\ln N)^{1+\lambda}<\infty$ for some $\lambda>4$ and $L$ obeys the law \eqref{SRW},
   then  $ \{ N_{q,n} \} (q=1,2,3,4)$  converge to   limit random variables (vectors)
	$\mathcal{V}_q(q=1,2,3,4)$ and $\mathcal{V}_2^z$  \mbox{ a.s.} with rates  $o(n^{-(\lambda-3)})$, respectively.
\end{theorem}
\begin{remark}
For  the simple branching random walk ($\zeta_0=0,\zeta_{1,1}=\cdots=\zeta_{d,1}=1/d $),	the fact that $\{N_{1,n}\} $ and $\{\inp{N_{2,n} }{ \mathbf{1}}\}$ are martingales was firstly observed by Chen  \cite{Chen2001}, where their a.s. convergences  are proved under the condition $\E N^2<\infty$. 
\end{remark}

Theorem \ref{thmCRM} is an immediate consequence of the following 
general result:

\begin{theorem}[\cite{Gao2018SPA}]\label{thmMG}Let $f$ be  a real (or vector-valued) function of $(x,n)\in \R^d\times \bZ$ satisfying $$\E f\big(x+(X-\mu),n+1\big)= f(x,n) \quad \mbox{ for $X$ a random vector with expectation $\mu$ } .$$  Consider a branching random walk with all $N_u$ and $L_u$ identically distributed  with $N$ and $X$, respectively. Then  
	 the associated sequence $\{\calM_n\}_{n\geq0} $ defined by 
	\begin{equation}\label{MartG}
	\calM_n= \dfrac{1}{m^n}\sum_{u\in \T_n} f\big(S_u-n\mu,n\big) 
	\end{equation}  
	is a martingale with respect to   the filtration $\{\D_{n}\}_{n\geq0} $.  

Assume that \eqref{SC} and \eqref{NE} hold. If 	given a constant $\rmQ\geq0$,  $ \E N(\ln N)^{1+\rmM} <\infty$ for some $M\geq \max\{2\rmQ-1,\rmQ\} $ and  $\E \norm{X}^{2\rmQ} <\infty$,   ,  
and the function $f$ satisfies the inequality
\begin{equation*}
\abs{f(x,n)} \leq C (\norme{x}^{2\rmQ}+n^\rmQ) \quad  \forall   (x,n)\in \R^d\times  \N, \  C  \mbox
	{ a constant},
\end{equation*}
then the martingale $\{\calM_n\}_{n\geq0} $ converges to a random element $\mathcal{V}$ a.s. and  
\begin{equation}\label{ConvMG}
\calM_n-\mathcal{V}=o(n^{-(\rmM-\max\{2\rmQ-1,\rmQ\} )})  \quad a.s.
\end{equation}
\end{theorem}
\begin{remark} Throughout the paper, we use $C$ to denote  various
	positive constants whose values are of no importance.
	
In the case that $f\equiv 1$ and $\rmQ=0$,  this theorem implies    Theorem 2 of \cite{Asmussen1976AOP}: 
\vskip 3mm
{\it If   $ \E N(\ln N)^{1+\rmM} <\infty$, then  a.s.
	\begin{equation}\label{ConvW}
	W_n-W=o(n^{- \rmM }). 
	\end{equation}}	
\end{remark}
Now let us go  
to deduce Theorem \ref{thmCRM}.   	
\begin{proof}[Proof of Theorem \ref{thmCRM}]
	We first prove that $\calV_4=\lim_{n\rightarrow \infty} N_{4,n}$.	
	Observe that for $u\in \T_n,  1\leq i\leq N_{u}$,
	\begin{align*}
	&\inp{S_{ui}}{\Gamma_2^{-1}S_{ui}}^2= \inp{S_u}{\Gamma_2^{-1}S_u}^2+ \inp{L_{ui}}{\Gamma_2^{-1}L_{ui}}^2+4\inp{L_{ui}}{\Gamma_2^{-1}S_u}^2\\
	& ~~~\qquad \qquad\qquad  +2 \inp{S_u}{\Gamma_2^{-1}S_u}\inp{L_{ui}}{\Gamma_2^{-1}L_{ui}} +4  \inp{S_u}{\Gamma_2^{-1}S_u}   \inp{L_{ui}}{\Gamma_2^{-1}S_u} 
 \\
	& ~~~\qquad \qquad\qquad + 4 \inp{L_{ui}}{\Gamma_2^{-1}L_{ui}}     \inp{L_{ui}}{\Gamma_2^{-1}S_u},                                          \\
	& \inp{S_{ui}}{\Gamma_2^{-1}S_{ui}}=    \inp{S_u}{\Gamma_2^{-1}S_u}+ \inp{L_{ui}}{\Gamma_2^{-1}L_{ui}}+ 2\inp{L_{ui}}{\Gamma_2^{-1}S_u}.  
	\end{align*} 
As all $L_{ui}$ obeys the law \eqref{SRW},  we  have 
	\begin{align*}
&	\E \inp{L_{ui}}{\Gamma_2^{-1}L_{ui}}^2 = \sum_{s=1}^d \sum_{r=1}^{t_s} \zeta_{s,r}\inp{r\e_s}{\Gamma_2 ^{-1}r\e_s}^2= \sum_{s=1}^d \sum_{r=1}^{t_s} \zeta_{s,r} r^4 \big( \zeta_s(2)\big)^{-2}
	\\
	&\qquad\qquad\qquad\quad= \sum_{s=1}^d  \zeta_s(4 ) \big( \zeta_s(2)\big)^{-2}=\tr{\Gamma_4\Gamma_2^{-2}}
\\
&\E \inp{L_{ui}}{\Gamma_2^{-1}L_{ui}} =\sum_{s=1}^d \sum_{r=1}^{t_s} \zeta_{s,r}\inp{r\e_s}{\Gamma_2 ^{-1}r\e_s} = \sum_{s=1}^d \sum_{r=1}^{t_s} \zeta_{s,r} r^2 \big( \zeta_s(2)\big)^{-1}=d,   \quad~~~\E L_{ui}=0,  
	\\ 
	& \E _{\D_{n}} \inp{L_{ui}}{\Gamma_2^{-1}S_u}^2=\inp{S_u}{\Gamma_2^{-1}S_u},
	\\ 
	& \E _{\D_{n}}  \inp{L_{ui}}{\Gamma_2^{-1}L_{ui}} \inp{L_{ui}}{\Gamma_2^{-1}S_u}=0.
	\end{align*}
	Then
	\begin{align*}
		&  \E_{\D_n} \inp{S_{ui}}{\Gamma_2^{-1}S_{ui}}= \inp{S_u}{\Gamma_2^{-1}S_u}+d,
		\\ &  \E_{\D_n} \inp{S_{ui}}{\Gamma_2^{-1}S_{ui}}^2=\inp{S_u}{\Gamma_2^{-1}S_u}^2  + (4+2d) \inp{S_u}{\Gamma_2^{-1}S_u}+\tr{\Gamma_4\Gamma_2^{-2}},   
	\end{align*}	
	Hence for the function $f$ defined by   
	$$f(x,n)= \inp{x}{\Gamma_2^{-1}x}^2 -(4+2d) n\inp{x}{\Gamma_2^{-1}x}+d(d+2)n^2+d(d+2)n- \tr{\Gamma_4\Gamma_2^{-2}}n,$$ 
the following relations hold:
\begin{equation*}
\E_{\D_{n}} f(S_{ui},n+1)=f(S_u,n), \quad |f(x,n)|\leq C(\norme{x}^4+n^2).
\end{equation*} 
 So	the conditions of  Theorem \ref{thmMG}  is satisfied  with $\rmQ=2$,   
	 therefore the sequences $(N_{4,n})_{n\geq 0}$ converges to some random variable $\calV_4$ with  rate $o(n^{-(\lambda-3)})$.	
	
	Similarly, we can prove that 
each  sequence 
  $(N_{q,n})_{n\geq 0}$ ( $q=1,2,3$) converges to some random variable $\calV_q$ with  rate $o(n^{-(\lambda-q+1)})=o(n^{-(\lambda-3)})$,  and $(N^z_{2,n})_{n\geq 0}$ to $\calV^z_2$ with  rate $o(n^{-(\lambda-3)})$. 
		To avoid repetition, we omit the details.
\end{proof}


\section{Second order expansion in LLT for  a finite range symmetric random walk} \label{Sec2}

In this section, we shall derive second order expansions in the local limit theorem for a finite range symmetric 
random walk on the integer lattice. This is a useful  complement of those results on  local limit theorem summarized in  \cite[Chapter 2]{Lawler2010}.

\begin{theorem}\label{Th2} Assume $\kappa\in (0,1/6)$  and    $L$ obeys the law \eqref{SRW}.     Let  $\{\widetilde{L}_n\}$ be  independent copies of $L$  and   set $\widetilde{S}_n = \widetilde{L}_1+ \cdots + \widetilde{L}_n$.   Then for $z=(z_1,z_2,\cdots,z_d)\in \Z^d$  , as $n\rightarrow \infty$,   
\\ (I) in the case   $L\in \calA_d$,
\begin{multline}\label{Conv-rate-rw}
\P (\widetilde{S}_n =z )= \dfrac{  (2\pi n)^{-d/2}}  {\sqrt{ \det \Gamma_2} }    \Bigg\lbrace 1+ \dfrac{1}{n}  \bigg[ \tau_d -\dfrac{1}{2} \inp{z}{\Gamma_2^{-1}z}  \bigg]   
\\ 
+ \dfrac{1}{n^2}	\bigg[ \frac{1}{8} \inp{z}{\Gamma_2^{-1}z}  ^2
- \inp{ \Lambda_d z}{ z}  +\chi_d\bigg]  \Bigg\rbrace 
+\frac{1}{n^{d/2+2}} \gamma_{n}(z),
\end{multline} 
\\ (II)   in the case  $L \in \calB_d $,  provided that  $n\equiv z_1+\cdots +z_d ~ (\mathrm{mod} ~2)$,
 	\begin{multline}\label{Conv-rate-lrw2}
 	\P (\widetilde{S}_n =z )= \dfrac{ 2(2\pi n)^{-d/2}}  {\sqrt{ \det \Gamma_2} }  
 	\Bigg\lbrace 1+ \dfrac{1}{n}  \bigg[ \tau_d -\dfrac{1}{2} \inp{z}{\Gamma_2^{-1}z}  \bigg]   
 	\\ 
 	+ \dfrac{1}{n^2}	\bigg[ \frac{1}{8} \inp{z}{\Gamma_2^{-1}z}  ^2
 	- \inp{ \Lambda_d z}{ z}  +\chi_d\bigg]  \Bigg\rbrace 
 	+\frac{1}{n^{d/2+2}} \gamma_{n}(z),
 	\end{multline}
	where the quantities $ \tau_d  $, $\chi_d$  and $\Lambda_d $  are  defined by  \eqref{eq-taud} \eqref{eq-Cd} and\eqref{eq-Lambdad}, 
	and $\gamma_{n}(z)$  is  infinitesimal satisfying
	\begin{equation}\label{small}
	\sup_{\{ \|z\| \leq  C n^{\kappa}  \} } |\gamma_{n}(z)|   \xrightarrow{n\longrightarrow \infty}
	0.
	\end{equation}
\end{theorem}

\begin{proof}
Consider the characteristic function $\psi$ of the random variable $L$, defined by the following 
	\begin{equation*}
	\psi(\phi ) = \E e^{\mathbf{i} \inp{\phi}{L} }=\zeta_0+ \sum_{s=1}^{d} \sum_{r=1}^{t_s}\zeta_{s,r}\cos(r\phi_s),\qquad  \phi=(\phi_1,\phi_2,\cdots,\phi_d)\in \R^d. 
	\end{equation*}  	 
	We recall  the following inversion formula (\cite[Corollary 2.2.3]{Lawler96}): 
	\begin{equation}\label{Eq-CF}
	\P (\widetilde{S}_n=z)= (2\pi)^{-d}\int_{[-\pi,\pi]^d}\cos \inp{\phi}{z}\psi(\phi)^n \rmd \phi.
	\end{equation}
	By  Taylor's expansion  about  0
	of $\psi (\phi)$,  
	\begin{equation*}
	\psi (\phi)= 1-  \frac{1}{2}\sum_{s=1}^d\sum_{r=1}^{t_s} \zeta_{s,r} r^2\phi_s^2+ O(\norme{\phi}^4)= 1-  \frac{1}{2}\sum_{s=1}^d  \zeta_{s} (2)\phi_s^2+ O(\norme{\phi}^4).
	\end{equation*}
	Then
	we  can find  an $r \in  (0,  \pi /2) $ such that
	\begin{equation*}
	|\psi (\phi)| \leq  1-\frac{1}{4} \zeta_{\diamond}\norme{\phi}^2\leq e^{- {\zeta_{\diamond}\norme{\phi}^2}/{4}},     \mbox{   for   }  \norme{\phi}\leq r,\quad \zeta_{\diamond}:= \min\{\zeta_1(2),\cdots,\zeta_d(2)\}.
	\end{equation*}
(I) First we consider the case $L \in \calA_d$.	 

By the assumption on $L$,  $\widetilde{S}_n$ is aperiodic,  therefore from \cite[Lemma 2.3.2]{Lawler2010},  
it follows that  
	\begin{equation*}
\rho:= \sup\{	|\psi(\phi)|:  \phi\in  [-\pi,\pi]^d, \norme{\phi}> r \}<1.
	\end{equation*}
	Hence
	\begin{equation}\label{SdBRWeq3-2}
	\P(\widetilde{S}_n=z)=
	\I_n(z)+
	\J_n(z),
	\end{equation}
	where
	\begin{equation*}
	\I_n(z)= (2\pi)^{-d}\int_{ \norme{\phi}\leq r} \cos \inp{\phi}{z} \Big(\psi (\phi) \Big) ^{n}
	d \phi, \qquad   \J_n(z) \leq    \rho^n. 
	\end{equation*}
	Hence we  have 
	\begin{equation*}
	\J_n(z) = \frac{1}{n^{2+d/2}} \alpha^e_n(z), \quad \mbox{with} \sup_{\{ \norme{z} \leq   C n^{\kappa}  \} } \abs{\alpha^e_n(z)} \leq   \rho^n n^{2+d/2} .
	\end{equation*}
	Changing the variable  $\theta= \sqrt{n}\phi$,  we get
	\begin{equation*}
	\I_n(z)= (2\pi)^{-d}  n^{-d/2}\int_{ \norme{\theta}\leq \sqrt{n}r} \cos\Big( \frac{1}{\sqrt{n}}\inp{\theta}{  z } \Big) \Big(\psi ( \frac{\theta}{\sqrt{n}}) \Big) ^{n}d \theta.
	\end{equation*}
	Take a constant {$\delta= {1}/{12}- \kappa/2$}.  We decompose  $\I_n$ as follows:
	\begin{equation}\label{SdBRWeq3-3}
	(2\pi\sqrt{n})^{ d} \I_n(z)= \I_{1,n}(z)+\I_{2,n}(z),
	\end{equation}
	where 
	\begin{align*}
	\I_{1,n}(z)&= \int_{n^{\delta}\leq \norme{\theta}\leq r\sqrt{n}}  \cos\Big( \frac{1}{\sqrt{n}}\inp{\theta}{ z } \Big) \Big(\psi ( \frac{\theta}{\sqrt{n}}) \Big) ^{n}d \theta,  \\
	\I_{2,n}(z)&=  \int_{ \norme{\theta}<n^{\delta} }  \cos\Big( \frac{1}{\sqrt{n}}\inp{\theta}{ z } \Big) \Big(\psi ( \frac{\theta}{\sqrt{n}}) \Big) ^{n}d \theta.
	\end{align*}
	Observe that 
	\begin{equation*}
	|\I_{1,n}(z)|   \leq    \int_{n^{\delta}\leq \norme{\theta}\leq r\sqrt{n}}   e^{- {\zeta_{\diamond}\norme{\theta}^2}/{ 4  }}  d\theta
	\leq (2r)^{d}n^{d/2}e^{-\zeta_{\diamond}n^{2\delta}/ 4  }.
	\end{equation*}
	Then we can write  that  
	\begin{equation}\label{SdBRWeq3-4}
	\I_{1,n}(z)=  \alpha^e_n(z) \frac{1}{n^{2}}\quad  \mbox{with}  \sup_{\{ \norme{z} \leq C n^{\kappa}  \} } \abs{\alpha^e_n(z)} \leq (2 r)^{d}n^{2+d/2}e^{-\zeta_{\diamond}n^{2\delta}/ 4  }.
	\end{equation}
	
	Next we pass to the analysis of the term $I_{2,n}$.
	
	Note $\delta+ \kappa<1/6$. For  $ \norme{\theta}< n^{\delta}$, $\norme{z}\leq C n^{\kappa}$,  , we have
	\begin{equation*}
	|\inp{\theta}{z} | \leq \norme{\theta}\norme{z}< C n^{\delta+\kappa}<n^{1/6}. 
	\end{equation*} 
	Then by Taylor's expansion, we have that   for $\norme{\theta}< n^{\delta},\norme{z}\leq C n^{\kappa} $,
	\begin{equation}\label{eq-cos}
	\cos\Big( \frac{1}{\sqrt{n}}\inp{\theta}{ z } \Big)= 1-\frac{1}{2n} \inp{\theta}{ z }^2+\frac{1}{24n^2} \inp{\theta}{z}^4  +  \gamma_{n}(\theta,z)\frac{1}{n^{2}} ,
	\end{equation}
	where     
\begin{equation*}
	\abs{\gamma_n(\theta,z)} \leq\dfrac{1}{720}\dfrac{\inp{\theta}{z}^6 }{  {n}}\leq C n^{- 1+6(\delta+\kappa)}.
\end{equation*} 
	By the definition of $\psi$ and  using   Taylor's expansion 
	\begin{equation*}
	\cos x= 1-\frac{1}{2}x^2+ \frac{1}{4!}x^4-\frac{1}{6!}x^6+  \frac{r_x}{8!}x^{8}, \quad  \mbox{ with } |r_x|\leq 1,  x\in \mathbb{R}, 
	\end{equation*}
	we  have 
	\begin{align*}
	\Big(\psi \big( \frac{\theta}{\sqrt{n}}\big) \Big) ^{n}=& \exp \bigg\{n\ln \bigg[ \zeta_0+ \sum_{s=1}^{d} \sum_{r=1}^{t_s}\zeta_{s,r}\cos\Big(r\frac{ \theta_s}{\sqrt{n}}\Big)\bigg]  \bigg\}\\
	=&   \exp \Bigg\{   n  \ln\Bigg[ \zeta_0+ \sum_{s=1}^d \sum_{r=1}^{t_s}\zeta_{s,r}\bigg( 1-\frac{r^2\theta_s^2}{2n} + \frac{r^4\theta_s^4}{24n^2} - \frac{r^6\theta_s^6}{6! n^3} + \beta_{s,n}(r\theta) \frac{1}{n^3}  \bigg)   \Bigg]\Bigg\} \\
	= &    \exp \Bigg\{  n  \ln\Big[  1-   \frac{1}{2n }\sum_{s=1}^{d} \zeta_s(2)  \theta_s^2 + \frac{1}{24n^2 }\sum_{s=1}^{d} \zeta_s(4)\theta_s^4 - \frac{1}{6! n^3  }\sum_{s=1}^{d}  \zeta_s(6)\theta_s^6 +  \beta_n(\theta) \frac{1}{n^3}      \Big]    \Bigg\}
	\end{align*}
	Here and below we use $\beta_n(\cdot)$ as an infinitesimal as $n$ tends to infinity, which may take different values even in the same line.
	Again using Taylor's expansion
	\begin{equation*}
	\ln(1+ w)=  w-\frac{1}{2}w^2+\frac{1}{3}w^3-\frac{r_w}{4}w^4,\quad  |r_w|< 1, |w|<1,
	\end{equation*}
	we have 
	\begin{align*}
	&\Big(\psi \big( \frac{\theta}{\sqrt{n}}\big) \Big) ^{n}= \exp \Bigg\{  n \bigg[  -   \frac{1}{2n }\sum_{s=1}^d \zeta_s (2) \theta_s^2  + \frac{1}{24n^2} \sum_{s=1}^d \zeta_s(2)\theta_s^4 - \frac{1}{6! n^3 }  \sum_{s=1}^d  \zeta_s(6)\theta_s^6 \\
	& \qquad - \frac{1}{2} \Big(  -   \frac{1}{2n }\sum_{s=1}^d \zeta_s(2)  \theta_s^2 + \frac{1}{24n^2 } \sum_{s=1}^d  \zeta_s(4)\theta_s^4  \Big)^2  +\frac{1}{3} \Big(-   \frac{1}{2n }\sum_{s=1}^d \zeta_s(2)  \theta_s^2 \Big)^3 +   \beta_n(\theta) \frac{1}{n^3} \bigg]\Bigg\}
	\\ &\quad= \exp  \bigg\{ -   \frac{1}{2}\sum_{s=1}^d \zeta_s(2)  \theta_s^2\bigg\}\exp  \Bigg\{  \frac{1}{n}  \bigg[    \frac{1}{24 } \sum_{s=1}^d \zeta_s(4)\theta_s^4 -\frac{1}{8 }{\Big(\sum_{s=1}^d \zeta_s (2) \theta_s^2\Big)^2 }\bigg]
	\\
	&\qquad     +\frac{1}{n^2} \bigg[  \frac{1}{48}{\sum_{s=1}^d \zeta_s(2)  \theta_s^2} \sum_{s=1}^d \zeta_s(4)\theta_s^4  -  \frac{1}{6! }  \sum_{s=1}^d \zeta_s(6)\theta_s^6- \frac{1}{24} {\Big(\sum_{s=1}^d \zeta_s(2)  \theta_s^2\Big)^3}\bigg] + \beta_n(\theta)\frac{1}{n^2} \Bigg\}.
	\end{align*}
	Once again using Taylor' expansion 
	\begin{equation*}
	e^{x}= 1+x+\frac{1}{2}x^2+ \frac{e^{r_x}}{3} x^3,\quad  |r_x|<|x|,  x\in \mathbb{R},  
	\end{equation*}
	we get 
	\begin{multline}\label{eq-psi}
	\Big(\psi \big( \frac{\theta}{\sqrt{n}}\big) \Big) ^{n}=\exp\bigg\{\!\!-\frac{1}{2}\sum_{s=1}^d \zeta_s(2)  \theta_s^2\bigg\} \Bigg\{  1+  \frac{1}{n}  \bigg[    \frac{1}{24 } \sum_{s=1}^d \zeta_s(4)\theta_s^4 -\frac{1}{8 }{\Big(\sum_{s=1}^d \zeta_s(2)  \theta_s^2\Big)^2 }\bigg]   
	\\
\qquad\qquad	+\frac{1}{n^2} \bigg[ - \frac{1}{24 } {\Big(\sum_{s=1}^d \zeta_s(2)  {\theta_j}^2\Big)^3}
	-  \frac{1}{6! }  \sum_{s=1}^d \zeta_s(6)\theta_s^6+  \frac{1}{48 }{\sum_{s=1}^d \zeta_s(2)  \theta_s^2} \sum_{s=1}^d \zeta_s(4)\theta_s^4 
	\\
	+ \frac{1}{2} \bigg(      \frac{1}{24 } \sum_{s=1}^d \zeta_s(4)\theta_s^4  -\frac{1}{8 }{\Big(\sum_{s=1}^d \zeta_s (2) \theta_s^2\Big)^2 }\bigg) ^2 \bigg]	
	+   \beta_n(\theta)\frac{1}{n^2}  \Bigg \},
	\end{multline}
	where all    $\beta_n(\theta)$ are infinitesimals satisfying  {$\sup_{\norme{\theta}\leq n^{\delta}}\abs{\beta_n(\theta)}\leq C n^{-1+10\delta}.$}
	
	Combining \eqref{eq-cos} and \eqref{eq-psi}, we get
	\begin{multline}\label{eq-cospsi}
	\cos\Big( \frac{1}{\sqrt{n}}\inp{\theta}{z} \Big) \Big(\psi ( \frac{\theta}{\sqrt{n}}) \Big) ^{n} 
	\\
	= \exp\bigg\{\!\!-\frac{1}{2}\sum_{s=1}^d \zeta_s(2)  \theta_s^2\bigg\} \Bigg\{  1+  \frac{1}{n}  \bigg[   \frac{1}{24 } \sum_{s=1}^d \zeta_s(4)\theta_s^4 -\frac{1}{8 }{\Big(\sum_{s=1}^d \zeta_s(2)   \theta_s^2\Big)^2 } - \frac{\inp{\theta}{z}^2}{2}\bigg]\\  +\frac{1}{n^2} \bigg[    \frac{1}{48 }{\sum_{s=1}^d \zeta_s (2) \theta_s^2} \sum_{s=1}^d \zeta_s(4)\theta_s^4 
	+ \frac{1}{2} \bigg(      \frac{1}{24 } \sum_{s=1}^d \zeta_s(4)\theta_s^4 -\frac{1}{8 }{\Big(\sum_{s=1}^d \zeta_s(2)  \theta_s^2\Big)^2 }\bigg) ^2 
	-  \frac{1}{6! }  \sum_{s=1}^d \zeta_s(6)\theta_s^6
	\\
- \frac{1}{24 } {\Big(\sum_{s=1}^d \zeta_s(2)  \theta_s^2\Big)^3} 
	-  \frac{\inp{\theta}{z}^2 }{2} \Big(\frac{1}{24 } \sum_{s=1}^d \zeta_s(4)\theta_j^4 -\frac{1}{8 }{\Big(\sum_{s=1}^d \zeta_s(2)  {\theta_s}^2\Big)^2 }\Big)   + \frac{\inp{\theta}{z}^4 }{24}   \bigg]  
	\\ +    \widetilde{\gamma}_n(\theta,z)   \frac{1}{n^2}      \Bigg\},
	\end{multline}
	where  the term $\widetilde{\gamma}_n(\theta,z)$ satisfies
	\begin{equation*}
	\sup_{\{ \norme{z} \leq C n^{\kappa}, \norme{\theta}\leq n^{\delta}  \} } \abs{ \widetilde{\gamma}_n(\theta,z)}\leq C n^{-1+ \max\{6(\delta+\kappa), 10\delta \}}.
	\end{equation*}  
	Then we see that 
	\begin{equation*}
	\sup_{\{ \norme{z} \leq C n^{\kappa}  \} }\abs{\int_{\norme{\theta}< n^{\delta}}     { \widetilde{\gamma}_n(\theta,z)} e^{- \frac{1}{2}\sum_{s=1}^d \zeta_s(2)\theta^2 } d\theta} \leq    C n^{-1+ \max\{6(\delta+\kappa),10\delta \}}.
	\end{equation*}
By elementary but tedious calculus, we obtain the following results:
	\begin{align*}
	\mathbf{1}. ~&  \int_{\norme{\theta}<n^{\delta}} \exp\bigg\{\!\!-\frac{1}{2}\sum_{s=1}^d \zeta_s(2)  \theta_s^2\bigg\} \inp{\theta}{z}^2\mathrm{d} \theta= \dfrac{ (2\pi)^{d/2}}{\sqrt{ \det \Gamma_2}} \inp{z}{\Gamma_2^{-1}z}+     \alpha^e_{n}(z)  \frac{1}{n^2},
	\\
	\mathbf{2}. ~&\int_{\norme{\theta}<n^{\delta}} \exp\bigg\{\!\!-\frac{1}{2}\sum_{s=1}^d \zeta_s(2)  \theta_s^2\bigg\} \inp{\theta}{z}^4\mathrm{d} \theta= \dfrac{ (2\pi)^{d/2}}{\sqrt{ \det \Gamma_2}} \cdot 3 \inp{z}{\Gamma_2^{-1}z}^2+     \alpha^e_{n}(z)  \frac{1}{n^2},
	\\
	\mathbf{3}. ~&\int_{\norme{\theta}<n^{\delta}} \exp\bigg\{\!\!-\frac{1}{2}\sum_{s=1}^d \zeta_s(2)  \theta_s^2\bigg\} \bigg(\sum_{s=1}^d \zeta_s (2)\theta_s^2\bigg)^2\inp{\theta}{z}^2\mathrm{d} \theta
	\\ 
	& \qquad \qquad\qquad  = \dfrac{ (2\pi)^{d/2}}{\sqrt{ \det \Gamma_2}} \cdot   (d+2)(d+4) \inp{z}{\Gamma_2^{-1}z}+     \alpha^e_{n}(z)  \frac{1}{n^2},
	\\
	\mathbf{4}. ~&\int_{\norme{\theta}<n^{\delta}} \exp\bigg\{\!\!-\frac{1}{2}\sum_{s=1}^d \zeta_s(2)  \theta_s^2\bigg\} \bigg( \sum_{s=1}^d \zeta_s (4)\theta_s^4  \bigg)\inp{\theta}{z}^2\mathrm{d} \theta 
	\\ & \qquad \qquad\qquad  = \dfrac{ (2\pi)^{d/2}}{\sqrt{ \det \Gamma_2}}\cdot 3\Big(4\inp{z}{\Gamma_4 \Gamma_2 ^{-3}z}+\tr{\Gamma_4 \Gamma_2^{-2} } \inp{z}{\Gamma_2^{-1}z}\Big ) +     \alpha^e_{n}(z)  \frac{1}{n^2},
	\\ 
\mathbf{5}.	~&\int_{\norme{\theta}<n^{\delta}} \exp\bigg\{\!\!-\frac{1}{2}\sum_{s=1}^d \zeta_s(2)  \theta_s^2\bigg\} \mathrm{d} \theta
	= \dfrac{ (2\pi)^{d/2}}{\sqrt{ \det \Gamma_2}}+\varepsilon^e_n  \frac{1}{n^2},  
	\\
\mathbf{6}.	~&\int_{\norme{\theta}<n^{\delta}} \exp\bigg\{\!\!-\frac{1}{2}\sum_{s=1}^d \zeta_s(2)  \theta_s^2\bigg\} \bigg( \sum_{s=1}^d \zeta_s(4) \theta_s^4  \bigg)\mathrm{d} \theta
	=\dfrac{ (2\pi)^{d/2}}{\sqrt{ \det \Gamma_2}} \cdot 3 \tr{\Gamma_4  \Gamma_2 ^{-2}}+\varepsilon^e_n  \frac{1}{n^2}  ,
	\\
\mathbf{7}.~	&\int_{\norme{\theta}<n^{\delta}} \exp\bigg\{\!\!-\frac{1}{2}\sum_{s=1}^d \zeta_s(2)  \theta_s^2\bigg\} \bigg(\sum_{s=1}^d \zeta_s(2) \theta_s^2\bigg)^2\mathrm{d} \theta
	= \dfrac{ (2\pi)^{d/2}}{\sqrt{ \det \Gamma_2}}\cdot  d(d+2) +\varepsilon^e_n  \frac{1}{n^2}  ,
	\\
\mathbf{8}.	~&\int_{\norme{\theta}<n^{\delta}} \exp\bigg\{\!\!-\frac{1}{2}\sum_{s=1}^d \zeta_s(2)  \theta_s^2\bigg\} \bigg( \sum_{s=1}^d \zeta_s (2)\theta_s^2  \sum_{s=1}^d \zeta_s(4) \theta_s^4\bigg)\mathrm{d} \theta
	\\
	&\qquad\qquad\qquad= \dfrac{ (2\pi)^{d/2}}{\sqrt{ \det \Gamma_2}} \cdot 3(d+4)\tr{\Gamma_4\Gamma_2^{-2}} +\varepsilon^e_n  \frac{1}{n^2}  ,
	\\ 
\mathbf{9}.	~&\int_{\norme{\theta}<n^{\delta}} \exp\bigg\{\!\!-\frac{1}{2}\sum_{s=1}^d \zeta_s(2)  \theta_s^2\bigg\} \bigg(   \sum_{s=1}^d \zeta_s (6)\theta_s^6\bigg) \rmd\theta
	=   \dfrac{ (2\pi)^{d/2}}{\sqrt{ \det \Gamma_2}}\cdot  15 \tr{ \Gamma_6\Gamma_2^{-3} }+\varepsilon^e_n  \frac{1}{n^2}  ,
	\\ 
\mathbf{10}.~	&\int_{\norme{\theta}<n^{\delta}}\exp\bigg\{\!\!-\frac{1}{2}\sum_{s=1}^d \zeta_s(2)  \theta_s^2\bigg\} \bigg( \sum_{s=1}^d \zeta_s (2)\theta_s^2   \bigg)^3\mathrm{d} \theta
	= \dfrac{ (2\pi)^{d/2}}{\sqrt{ \det \Gamma_2}}\cdot d(d+2)(d+4)+\varepsilon^e_n  \frac{1}{n^2}  ,
	\\ 
\mathbf{11}.~	&
	\int_{\norme{\theta}<n^{\delta}}\exp\bigg\{\!\!-\frac{1}{2}\sum_{s=1}^d \zeta_s(2)  \theta_s^2\bigg\} \bigg(  \sum_{s=1}^d \zeta_s (2)\theta_s^2\bigg)^4\mathrm{d} \theta
\\
&\qquad\qquad\qquad
	= \dfrac{ (2\pi)^{d/2}}{\sqrt{ \det \Gamma_2}} \cdot d(d+2)(d+4)(d+6)+\varepsilon^e_n  \frac{1}{n^2}    ,
	\\
\mathbf{12}.~	&  
	\int_{\norme{\theta}<n^{\delta}} \exp\bigg\{\!\!-\frac{1}{2}\sum_{s=1}^d \zeta_s(2)  \theta_s^2\bigg\} \bigg( \sum_{s=1}^d \zeta_s (4)\theta_s^4   \bigg)^2\mathrm{d} \theta
	\\
	&\qquad\qquad\qquad
	= \dfrac{ (2\pi)^{d/2}}{\sqrt{ \det \Gamma_2}}  \cdot\Big[   96 \tr{ \Gamma_4^2\Gamma_2^{-4} } + 9 \big(\tr{\Gamma_4\Gamma_2^{-2}}\big)^2\Big]+\varepsilon^e_n  \frac{1}{n^2}   ,
	\\  
\mathbf{13}.~	&
	\int_{\norme{\theta}<n^{\delta}} \exp\bigg\{\!\!-\frac{1}{2}\sum_{s=1}^d \zeta_s(2)  \theta_s^2\bigg\} \bigg( \sum_{s=1}^d \zeta_s (4)\theta_s^4  \bigg)\bigg(\sum_{s=1}^d \zeta_s(2) \theta_s^2\bigg)^2\mathrm{d} \theta
	\\
	& \qquad\qquad \qquad
	= \dfrac{ (2\pi)^{d/2}}{\sqrt{ \det \Gamma_2}}   \cdot 3(d+4) (d+6)\tr{\Gamma_4\Gamma_2^{-2}}+\varepsilon^e_n  \frac{1}{n^2}. 
	\end{align*}
	In the above estimates, the infinitesimals $ \alpha^e_{n}(x)$  and  $\varepsilon^{e}_n$ are uniformly bounded by  $ C   e^{-n^{ \delta/2} }$,  that is, 
	\begin{equation*}
	\sup_{\norme{z}\leq C n^{\kappa}}|\alpha^e_{n}(z)| \leq  C   e^{-n^{ \delta/2} },\qquad   \varepsilon^e_n\leq C   e^{-n^{ \delta/2} }. 
	\end{equation*}    
	Recall that  
	\begin{equation*}
	\I_{2,n}(z) =  \int_{ \norme{\theta}<n^{\delta} }  \cos\Big( \frac{1}{\sqrt{n}}\inp{\theta}{ z } \Big) \Big(\psi ( \frac{\theta}{\sqrt{n}}) \Big) ^{n}d \theta.
	\end{equation*}
	By using the above integrals   and \eqref{eq-cospsi},  we  have 
	\begin{multline*} 
	\I_{2,n} (z)=  \dfrac{ (2\pi)^{d/2}}{\sqrt{ \det \Gamma_2}}  \Bigg\lbrace 1+ \dfrac{1}{n}  \bigg[\dfrac{1}{8}\Big(\mathrm{tr} (\Gamma_4\Gamma_2^{-2}) -d(d+2)\Big)   -\dfrac{1}{2} \inp{z}{\Gamma_2^{-1}z}  \bigg]   + \dfrac{1}{n^2}	\bigg[ \frac{1}{8} \inp{z}{\Gamma_2^{-1}z}  ^2
	\\ 
	- \dfrac{1}{16}\Big(  \tr{\Gamma_4\Gamma_2^{-2}} -(d+2)(d+4)\Big)\inp{\Gamma_2^{-1}z}{z} -\frac{1}{4}\inp{ \Gamma_4 \Gamma_2^{-3} z}{ z}       -\dfrac{1}{64}(d+2)(d+4)\tr{\Gamma_4\Gamma_2^{-2}} \\+ \frac{1}{12}\tr{\Gamma_4^2\Gamma_2^{-4}}+\frac{1}{128}\Big(\tr{\Gamma_4\Gamma_2^{-2}}\Big)^2 -\frac{1}{48}\tr{\Gamma_6\Gamma_2^{-3}}
	+\dfrac{1}{384}d(d+2)(d+4) (3d+2) \bigg]  + \frac{1}{n^{2}}   \alpha_{n}(z)\Bigg\rbrace
	,
	\end{multline*}
	where the term $\alpha_n(z) $ satisfies  
	\begin{equation}\label{eq-r}
	\sup_{\{ \norme{z} \leq  C n^{\kappa}  \} } \abs{\alpha_n(z) }\leq C n^{-1+ \max\{6(\delta+\kappa),10\delta \}}.
	\end{equation}
	With the notation $\tau_d$, $\Lambda_d$ and $\chi_d$ defined in \eqref{eq-taud}, \eqref{eq-Lambdad} and \eqref{eq-Cd},
	we get 
	\begin{multline}\label{SdBRWeq2-8}
	\I_{2,n} (z)=  \dfrac{ (2\pi)^{d/2}}{\sqrt{ \det \Gamma_2}}  \Bigg\lbrace 1+ \dfrac{1}{n}  \bigg[ \tau_d   -\dfrac{1}{2} \inp{z}{\Gamma_2^{-1}z}  \bigg]   
	+ \dfrac{1}{n^2}	\bigg[ \frac{1}{8} \inp{z}{\Gamma_2^{-1}z}  ^2
	\\ 
	-\inp{ \Lambda_d z}{ z}     +\chi_d \bigg]  + \frac{1}{n^{2}}   \alpha_{n}(z)\Bigg\rbrace,
	\end{multline}
	where the term $\alpha_n(z) $ satisfies   \eqref{eq-r}.
	
	Combining  the equations
	\eqref{SdBRWeq3-2} --\eqref{SdBRWeq2-8}, we obtain the desired result  \eqref{Conv-rate-rw}.  
\\
(II) Now we deal with the case   $L \in \calB_d$. 

In this case, it holds that  $\zeta_0=0$ and 
the set $\{ r:  \zeta_{s,r}>0,  s=1,2,\cdots,d\}$  only contains odd numbers.
Thus  the integrand of \eqref{Eq-CF} satisfies 
\begin{equation*}
\cos \inp{\phi+\pi\mathbf{1} }{z}\psi(\phi+\pi\mathbf{1})^n= (-1)^{n+ z_1+\cdots +z_d}\cos \inp{\phi}{z}\psi(\phi)^n =\cos \inp{\phi}{z}\psi(\phi)^n.
\end{equation*}
Therefore  \eqref{Eq-CF}  turns out to be 
\begin{equation}\label{Eq-bi}
\P (\widetilde{S}_n=z)= 2(2\pi)^{-d}\int_{A}\cos \inp{\phi}{z}\psi(\phi)^n \rmd \phi,
\end{equation}
where $A=[-\pi/2,\pi/2] \times [-\pi,\pi]^{d-1}$.
Observe that 
\begin{equation*}
\tilde{\rho}:= \sup\{	|\psi(\phi)|:  \phi\in  A, \norme{\phi}> r \}<1.
\end{equation*}
Hence
\begin{multline}\label{SdBRWeq3-2a}
\P(\widetilde{S}_n=z)= \tilde{I}_n + \tilde{J}_n
:= 2 (2\pi)^{-d}\int_{ \norme{\phi}\leq r} \cos \inp{\phi}{z} \Big(\psi (\phi) \Big) ^{n}
 d \phi
 \\+   2 (2\pi)^{-d}\int_{ A\backslash \{ \phi: \norme{\phi}\leq r\}} \cos \inp{\phi}{z} \Big(\psi (\phi) \Big) ^{n}
 d \phi,
\end{multline}
where $|\J_n(z)| \leq    \tilde{\rho}^n.$
Then following the same procedures as in (I), we can obtain  the desired \eqref{Conv-rate-lrw2}.
 \end{proof}

\section{Proofs of   Theorem  \ref{SBRW-Th}  and Corollary \ref{SBRW-Cor}}\label{Sec4}

We first introduce some notation.
As usual, we write $\N^* = \{1,2,3,\cdots \}$ and denote by
$$ U= \bigcup_{n=0}^{\infty} (\N^*)^n $$
the set of all finite sequences, where $(\N^*)^0=\{\varnothing \}$ contains the null sequence $ \varnothing$.

For all $u\in U$, let $\mathbb{T}(u)$ be the shifted tree of $\mathbb{T}$ at $u$  with defining elements $\{N_{uv}\}$: we have
1) $\varnothing \in \mathbb{T}(u)$, 2) $vi\in \mathbb{T}(u)\Rightarrow v\in \mathbb{T}(u)$ and  3) if  $v\in \mathbb{T}(u)$, then $vi\in \mathbb{T}(u)$ if and only if $1\leq i\leq N_{uv} $. Define $\mathbb{T}_n(u)=\{v\in \mathbb{T}(u): |v|=n\}$. Recall that $\mathbb{T}=\mathbb{T}(\varnothing)$ and $\mathbb{T}_n=\mathbb{T}_n(\varnothing)$.

Let $\kappa$ be a real number  satisfying
$  \frac{d+6}{2\lambda} <\kappa<\frac{1}{6}
$
and  set $k_n=\lfloor n^{\kappa}\rfloor$, the largest integer not bigger than $n^{\kappa}$.

From the additivity property of the branching process, it follows that 
\begin{equation}\label{dec}
Z_n(z)= \sum_{u\in \T_{k_n} } \sum_{v\in \mathbb{T}_{n-k_n}(u) } \ind{S_{uv}=z},
\end{equation}
and \begin{equation*}
\E_{\D_{k_n} } \left( \sum_{v\in \mathbb{T}_{n-k_n}(u) } \ind{S_{uv}=z} \right)
=m^{n-k_n}  \P\Big(\widetilde{S}_{n-k_n}=z-y\Big)\Big|_{y=S_u}. 
\end{equation*}
Then 
we have the following decomposition:
\begin{multline}\label{SBRW-eq4}
\frac{Z_n(z)}{ m^n  } =\frac{1}{ m^{k_n} } \sum_{u\in \T_{k_n}}  \Bigg(  \frac{\sum_{v\in \mathbb{T}_{n-k_n}(u) } \ind{S_{uv}=z} }{m^{n-k_n}} -  \P(\widetilde{S}_{n-k_n}=z-y)\Big|_{y=S_u}  \Bigg) \\ +  \frac{1}{ m^{k_n} } \sum_{u\in \T_{k_n}} \P(\widetilde{S}_{n-k_n}=z-y)\Big|_{y=S_u}=:  \mathbb{D}_{1,n} + \mathbb{D}_{2,n}.
\end{multline}


So Theorem  \ref{SBRW-Th}  will
 follows from the following two  lemmas:

\begin{lemma}\label{SBRW-Lem1} 
	Assume  that the conditions \eqref{SC} and \eqref{NE} hold,    $ \E N(\ln N)^{1+\lambda}<\infty$ for some  $\lambda>3(d+6)$ and $L$ has the law  \eqref{SRW}. Then we have
	\begin{equation}\label{SBRW-eq5}
	n^{2+\frac{d}{2}}\mathbb{D}_{1,n}  \xrightarrow{n \rightarrow \infty } 0 \quad   \mbox{ a.s.}
	\end{equation}
\end{lemma}
\begin{lemma}\label{SBRW-Lem2}
	Assume  that the conditions \eqref{SC} and \eqref{NE} hold,    $ \E N(\ln N)^{1+\lambda}<\infty$ for some  $\lambda>3(d+6)$ and $L$ has the law  \eqref{SRW}. Then  as $n \rightarrow \infty$,   a.s. 
\\
(I) when $L \in \calA_d$, 	
\begin{equation}\label{SBRW-eq6}
\mathbb{D}_{2,n} =\dfrac{ (2\pi n)^{-d/2}}{ \sqrt{ \det \Gamma_2} }    
\bigg [W + \frac{1}{n}  F_1(z) +\frac{1}{n^2} F_2(z) \bigg]+  \frac{1}{n^{2+d/2}}o(1);
\end{equation}
(II) when $L \in \calB_d$,  provided $ n\equiv z_1+z_2+\cdots+z_d\, (\,\mathrm{ mod }~  2)$, 
\begin{equation}\label{SBRW-eq6a}
\mathbb{D}_{2,n} =\dfrac{2(2\pi n)^{-d/2}}{ \sqrt{ \det \Gamma_2} }    
\bigg[   W + \frac{1}{n}  F_1(z)+\frac{1}{n^2}  F_2(z)  \bigg]  +  \frac{1}{n^{2+d/2}}o(1),   
\end{equation}
	where   $F_1(z)$ and $F_2(z)$ are defined by \eqref{f1}   and \eqref{f2}  respectively.

\end{lemma}

\begin{proof}[Proof of Theorem \ref{SBRW-Th}]
	Combining Lemmas \ref{SBRW-Lem1} with \ref{SBRW-Lem2}, together with \eqref{SBRW-eq4}, we obtain immediately Theorem \ref{SBRW-Th}.
\end{proof}

\begin{proof}[Proof of Corollary \ref{SBRW-Cor}]
This corollary  follows immediately from  Theorem \ref{SBRW-Th} by some elementary calculations.
\end{proof}

\begin{proof}[Proof of Lemma \ref{SBRW-Lem1}]
	We start by introducing some  notation.  For $u\in \T_{k_n}$, set
	\begin{eqnarray*}
		& &   X_{n,u}= \frac{\sum_{v\in \mathbb{T}_{n-k_n}(u) } \ind{S_{uv}=z} }{m^{n-k_n}} -  \P(\widetilde{S}_{n-k_n}=z-y)|_{y=S_u} ,    \quad \overline{X}_{n,u}=   X_{n,u} \mathbf{1}_{\{ |X_{n,u}|\leq m^{k_n}\}},\\
		& &   \overline{\mathbb{A}}_n = \frac{1}{m^{k_n}} \sum_{u\in \T_{k_n}} \overline{X}_{n,u}.
	\end{eqnarray*}
	It is easy to see the following fact:
	\begin{equation}\label{Dom-X}
	| X_{n,u}|\leq  W_{n-k_n}(u) +1, \quad  \mbox{ with }  W_{n-k_n}(u) =m^{-(n-k_n)}   \sum_{v\in \mathbb{T}_{n-k_n}(u) }1  . 
	\end{equation}
	We remind that $\{W_{n-k_n}(u): u\in \T_{k_n}\}  $ are mutually independent and identically distributed as $W_{n-k_n}$.
	
	The lemma will be proved if we can  show the following:
	\begin{align}
	\label{Zd-An-1}    &  \P ( \overline{\mathbb{A}}_n  \neq  \mathbb{A} _n \mbox{ i.o.} ) =0.  \\
	\label{Zd-An-2}  &{n^{d/2+2}}\Big(\overline{\A}_n -\E_{\D_{k_n}} \overline{\A}_n \Big)\xrightarrow{n \rightarrow \infty } 0\quad\mbox{ a.s. }   \\
	\label{Zd-An-3}   &   {n^{d/2+2}}\E_{\D_{k_n}} \overline{\A}_n \xrightarrow{n \rightarrow \infty } 0\quad\mbox{ a.s. }
	\end{align}
	
	To this end, we shall need the following result.
	\begin{lemma}(\cite{BinghamDoney1974AAP})\label{LTABWTlem4}
		Let $ W^*= \sup_n W_n$. 	Assume  $m>1$ and $ \E N(\log  N)^{ 1+\lambda}<\infty $.  Then
		\begin{equation}\label{Zd1.14}
		{\E}(W^*+1)(\log  (W^*+1))^{\lambda} <\infty.
		\end{equation}
	\end{lemma} 	
	To prove  \eqref{Zd-An-1},    it suffices to show that
	\begin{equation}\label{BC1}
	\sum_{n=1}^\infty \P(  \overline{\mathbb{A}}_n  \neq  \mathbb{A} _n  ) <\infty.
	\end{equation}
	Observe that
	\begin{align*}
	\P (\A_n\neq \overline{\A}_n )&\leq\E \sum_{u\in \T_{k_n} } \P_{ \D_{k_n}}(X_{n,u}\neq \overline{X}_{n,u})  =\E  \sum_{u\in \T_{k_n} } \P_{ \D_{k_n}}(|X_{n,u}|\geq m^{k_n}) \\&\leq_{\eqref{Dom-X}} \E \sum_{u\in \T_{k_n} }  \P ( W_{n-k_n} (u)+1 \geq m^{k_n})= m^{k_n} \P ( W_{n-k_n}  +1 \geq m^{k_n})
	\\&\leq   \E\big((W_{n-k_n}+1 ) \mathbf{1}_{ \{W_{n-k_n}+1 \geq m^{k_n}\}} \big)
	\\ &\leq \E\big((W^*+1 ) \mathbf{1}_{ \{W^*+1 \geq m^{k_n}\}} \big)
	\\& \leq  (\log m)^{-\lambda}  k_n ^{-\lambda}\E (W^*+1) \log ^{\lambda} (W^*+1) .
	\end{align*}
	Then \eqref{BC1} follows from  the choice of $k_n$,  the  fact $\lambda \kappa>1$  and  \eqref{Zd1.14}.

	Now we turn to the proof of  \eqref{Zd-An-2}.  To this end, we will need the following inequality (by (5.3) in \cite{Biggins79}):  for $1<\alpha<2$,
	\begin{align*}
	&\P_{ \D_{k_n}} \left[  \bigg|  \sum_{u\in \T_{k_n}}  {m^{-k_n}}(\bar{X}_{n,u}- \E_{\D_{k_n} } \bar{X}_{n,u}) \bigg|     >  \epsilon n^{-d/2-2}\right ] 
	\\ 
	\leq \,&  K \dfrac{n^{(d/2+2)\alpha}}{\epsilon^{\alpha}} \left\{ m^{ -\alpha k_n}   Z_{k_n}(\Z^d)  \E  (W^*+1) ^\alpha \ind{|W^*+1| \leq m^{k_n} } + Z_{k _n}(\Z^d) \E  \ind{|W^*+1| > m^{k_n} }         \right\}.
	\end{align*}
Since  $\E Z_{k_n}(\Z^d)=m^{k_n}$, 	taking expectation in the both sides of the above     gives  the following 
\begin{multline*}
\E \P_{ \D_{k_n}} \left[  \bigg|  \sum_{u\in \T_{k_n}}  {m^{-k_n}}(\bar{X}_{n,u}- \E_{\D_{k_n} } \bar{X}_{n,u}) \bigg|     >  \epsilon n^{-d/2-2}\right ] 
\\
\leq 
K\dfrac{n^{(d/2+2)\alpha}}{\epsilon^{\alpha}} \left\{ m^{ (1-\alpha) k_n}      \E  (W^*+1) ^\alpha \ind{|W^*+1| \leq m^{k_n} } +m^{k_n}  \E  \ind{|W^*+1| > m^{k_n} }         \right\}.
\end{multline*}	
	Note that in the above formula and throughout the paper, $K$ denotes all constants, and thus its value may vary even in a single inequality.  	
	Thus    by taking expected value of  the above, we deduce that 	
	\begin{eqnarray*}
		& & \sum_{n=1}^{\infty} {\P} (|\overline{\A}_n-{\E}_{\D_{k_n}} \overline{\A}_n|>{\varepsilon}{{n}^{-(d/2+2)}})  \\
		&=& \sum_{n=1}^{\infty}\E\P_{ \D_{k_n}} \left[  \bigg|  \sum_{u\in \T_{k_n}}  {m^{-k_n}}(\bar{X}_{n,u}- \E_{\D_{k_n} } \bar{X}_{n,u}) \bigg|     >  \epsilon/n^{d/2+2}\right ]  \\
		&\leq&  \sum_{n=1}^{\infty}
		K \dfrac{n^{(d/2+2)\alpha}}{\epsilon^{\alpha}} \left\{ m^{ (1-\alpha) k_n}      \E  (W^*+1) ^\alpha \ind{|W^*+1| \leq m^{k_n} } + m^{k _n} \E  \ind{|W^*+1| > m^{k_n} }     \right\}
		\\ &=& 	 K 	{\epsilon^{-\alpha}} \E    \left\{ (W^*+1) ^\alpha \sum_{n=1}^{\infty} {n^{(d/2+2)\alpha}}  m^{ (1-\alpha) k_n} \ind{|W^*+1| \leq m^{k_n} } + \sum_{n=1}^{\infty}   n^{(d/2+2)\alpha} m^{k _n}  \ind{|W^*+1| > m^{k_n} }\right\}
		\\ &\leq&  K 	{\epsilon^{-\alpha}}   \E  (W^*+1) \Big(\log (W^*+1)  \Big)^{((d+4)\alpha +2)/(2\kappa) -1},
	\end{eqnarray*}
	which is finite,   since $ ( (d+4)\alpha +2)/(2\kappa) -1 <\lambda $  provided that $\alpha$  is sufficiently near one
	and    $\E  (W^*+1) \big(\log (W^*+1)  \big)^{ \lambda} <\infty$.
	Hence   \eqref{Zd-An-2}   follows  by the Borel-Cantelli lemma.
	
	It remains to prove \eqref{Zd-An-3}.   Since $\E_{\D_{k_n}} X_{n,u}=0$, we see  a.s.
	\begin{eqnarray*}
		|\E_{\D_{k_n} } \overline{\A}_{n} |&=&  \bigg|\frac{1}{m^{k_n}} \sum_{u\in \T_{k_n}}\E_{\D_{k_n} }  \overline{X}_{n,u} \bigg| =\bigg|-\frac{1}{m^{k_n}} \sum_{u\in \T_{k_n}}\E_{\D_{k_n} }X_{n,u} \mathbf{1}_{\{ |X_{n,u}|\geq m^{k_n}\}}\bigg | \\
		&\leq& \frac{1}{m^{k_n}} \sum_{u\in \T_{k_n}}\E_{\D_{k_n} } (W_{n-k_n}(u)+1)  \mathbf{1}_{\{ |W_{n-k_n}(u)+1|\geq m^{k_n}\}}
		\\ &=& W_{k_n} \E ( W_{n-k_n}+1)  \mathbf{1}_{\{ |W_{n-k_n}+1|\geq m^{k_n}\}} \\
		&\leq & W^* \E (W^*+1)  \mathbf{1}_{\{ |W^*+1|\geq m^{k_n}\}}
		\\ &\leq& k_n^{-\lambda} (\log m) ^{-\lambda}  W^*  \E (W^*+1)  \log ^{\lambda}  (W^*+1).
	\end{eqnarray*}
	Combining this  with  the fact   $\lambda\kappa -d/2-2>1$,    we get  
	\begin{equation*}
	\Big |\sum_{n=1}^{\infty} {n^{d/2+2}}\E_{\D_{k_n}} \overline{\A}_{n}\Big | \leq \sum_{n=1}^{\infty}{n^{d/2+2}}k_n^{-\lambda}  W^*  \E (W^*+1)  \log ^{\lambda}  (W^*+1)   <\infty   \mbox{  a.s.}   
	\end{equation*}
	This implies the a.s. convergence of the series  $\sum_{n=1}^{\infty} {n^{d/2+2}}\E_{\D_{k_n}} \overline{\A}_{n}$, and accordingly \eqref{Zd-An-3} follows.
	The lemma has been proved.
\end{proof}
 
\begin{proof}[Proof of Lemma \ref{SBRW-Lem2}]
(I)  First we consider  the case $L\in \calA_d$.
	
Since $ \norme{S_u}\leq C k_n $  for $u\in \T_{k_n}$,  then by Theorem  \ref{Th2} (I), we have 
	\begin{multline}\label{eq-2e}
	\P(\widetilde{S}_{n-k_n}=z-y)\bigg|_{y=S_u} 
 =
	\dfrac{ \big(2\pi \big)^{-d/2}}{ \sqrt{\det \Gamma_2}    } \Bigg\lbrace \frac{1}{(n-k_n)^{d/2}}+ \dfrac{1}{(n-k_n)^{1+d/2}}  \bigg[\tau_d -\dfrac{1}{2} \inp{z-S_u}{\Gamma_2^{-1}(z-S_u)}  \bigg]   
	\\ 
 + \dfrac{1}{(n-k_n)^{2+d/2}}	\bigg[ \frac{1}{8} \inp{z-S_u}{\Gamma_2^{-1}(z-S_u)}  ^2 -\inp{\Lambda_d(z-S_u)}{z-S_u} +\chi_d\bigg]  \Bigg\rbrace 
\\ +\frac{\alpha_{n}(z,S_u) }{n^{d/2+2}} ,
	\end{multline}
where  $\tau_d$, $\chi_d$ and $\Lambda_d$ are defined by \eqref{eq-taud} \eqref{eq-Cd} and \eqref{eq-Lambdad},  and  $ \alpha_n(z,S_u) $ are infinitesimals such that
\begin{equation*}
\sup_{   u\in \T_{k_n}   } \abs{ \alpha_{n}(z,S_u) }\xrightarrow{n\rightarrow \infty} 0.
\end{equation*}
Observe that  the following relations hold:
	\begin{align*}
\mathbf{1}.~ 	& \inp{z-S_u }{ \Gamma_2^{-1} (z-S_u)} = \inp{z}{\Gamma_2^{-1}z} -2\inp{\Gamma_2^{-1}z}{S_u} + \inp{S_u}{ \Gamma_2^{-1}  S_u  }, 
\\
\mathbf{2}.~ 	&    \inp{z-S_u }{ \Gamma_2^{-1} (z-S_u)} ^2=\inp{z}{\Gamma_2^{-1}z}^2 +4\inp{\Gamma_2^{-1}z}{S_u}^2 + \inp{S_u}{ \Gamma_2^{-1}  S_u  } ^2 
\\ 
&\quad \quad\quad+2  \inp{z}{\Gamma_2^{-1}z}\inp{S_u}{ \Gamma_2^{-1}  S_u  } -4\inp{z}{\Gamma_2^{-1}z} \inp{\Gamma_2^{-1}z}{S_u} -4\inp{\Gamma_2^{-1}z}{S_u}   \inp{S_u}{ \Gamma_2^{-1}  S_u  },
	\\
\mathbf{3}.~ 	 
&  (n-k_n)^{-d/2}=\frac{1}{n^{d/2}}\bigg[1+ \frac{dk_n}{2n}  + \frac{d(d+2)}{8} \frac{k_n^2}{n^2}\bigg] + O( \frac{k_n^3}{n^{3+d/2 }} ),
	\\ 
\mathbf{4}.~ 	
&   (n-k_n)^{-1-d/2}= \frac{1}{n^{d/2}} \Bigg[ \frac{1}{n}   +   \frac{(d+2)k_n}{2n^2} + O(   \frac{k_n^2}{n^3 } )\Bigg] ,
	\\ 
\mathbf{5}.~ 
	&
	(n-k_n)^{-2}=\frac{1}{n^{2}} +   O(\frac{k_n}{n^3 }) .
	\end{align*}
By the definitions of  the quantities $\tau_d,\Lambda_d$, we can get 
 	\begin{equation}\label{eq-x }
 \tr{\Gamma_2\Lambda_d}= \dfrac{1}{16}d\Big(\tr{\Gamma_4\Gamma_2^{-2}}-(d+2)(d+4)\Big)+\frac{1}{4} \tr{ \Gamma_4 \Gamma_2^{-2}},
 \end{equation}
 hence \begin{equation}\label{eq-key}
 \Big(\frac{d}{2}+1\Big) k_n \tau_d= k_n \tr{\Gamma_2 \Lambda_d} + \frac{1}{8}\Big(d(d+2) - \tr{\Gamma_4\Gamma_2^{-2}} \Big)k_n.
 \end{equation}
 	Substituting the above expressions into \eqref{eq-2e}, we obtain  that 
 	\begin{align}
	\nonumber	  &\P\Big(\widetilde{S}_{n-k_n}=z-y\Big)\bigg|_{y=S_u}
		 \\ 
		 \nonumber =&~ 
		  \dfrac{ (2\pi   n)^{-d/2}}{\sqrt{ \det \Gamma_2}}      \times  \Bigg\lbrace\Big(1+\frac{dk_n}{2n} + \frac{d(d+2)k_n^2}{8n^2 }\Big )  + 
		  \\ 
		  \nonumber  &   \quad  \Big(\frac{1}{n}+ \frac{ (d+2)k_n}{2n^2}\Big) \bigg[\tau_d    -\dfrac{1}{2}\inp{z}{\Gamma_2^{-1}z} + \inp{S_u}{ \Gamma_2^{-1}z} - \frac{1}{2}\inp{S_u}{\Gamma_2^{-1}S_u}  \bigg]   
	 \\
	 \nonumber &\quad	+ \dfrac{1}{n^2}	\bigg[ \frac{1}{8}  \Big(  \inp{z}{\Gamma_2^{-1}z}^2 +4\inp{\Gamma_2^{-1}z}{S_u}^2 + \inp{S_u}{ \Gamma_2^{-1}  S_u  } ^2 +2  \inp{z}{\Gamma_2^{-1}z}\inp{S_u}{ \Gamma_2^{-1}  S_u  }
	\\ 
	\nonumber & -4\inp{z}{\Gamma_2^{-1}z} \inp{\Gamma_2^{-1}z}{S_u} -4\inp{\Gamma_2^{-1}z}{S_u}   \inp{S_u}{ \Gamma_2^{-1}  S_u  }\Big)  
	- 
	\Big(\inp{z}{\Lambda_d z}  	-2\inp{ S_u}{ \Lambda_d z}
	\\
\nonumber	&   +  \inp{S_u}{\Lambda_d  S_u}
	    \Big) 
	 		+ \chi_d  \bigg]  \Bigg\rbrace 
   +\alpha_{n,u} \frac{1}{n^{ 2+d/2}}  
 		\\ 
 		\nonumber=~&
 		 \dfrac{ (2\pi   n)^{-d/2}}{\sqrt{ \det \Gamma_2}}   \Bigg\{    1+\frac{1}{n} \bigg[\tau_d  -\dfrac{1}{2}\inp{z}{\Gamma_2^{-1}z} + \inp{S_u}{ \Gamma_2^{-1}z}  
 	 - \frac{1}{2}\Big(\inp{S_u}{\Gamma_2^{-1}S_u}    -dk_n \Big) \bigg]  
 	 \\ 
 \nonumber	 & + \frac{1}{n^2}
 		\bigg[  \Big(\dfrac{1}{8}\inp{z}{\Gamma^{-1}z}^2-    \inp{z}{\Lambda_d z}+ \chi_d \Big) 
 		+  \inp{ S_u}{\Big( 2\Lambda_d   - \frac{1}{2} \inp{z}{\Gamma_2^{-1}z} 
 			\Gamma_2^{-1}\Big) z}  
 		\\ 
 		\nonumber& - \bigg(  \inp{S_u}{\Big( \Lambda_d  - \frac{1}{4}  \inp{z}{\Gamma_2^{-1}z}\Gamma_2^{-1}\Big)  S_u} - k_n \Big(\tr{ \Gamma_2 \Lambda_d }- \frac{1}{4} d\inp{z}{\Gamma_2^{-1}z}\Big)\bigg)		
  \\ 
  \nonumber &		+\dfrac{1}{2} \Big( \inp{\Gamma_2^{-1}z}{S_u}^2- k_n\inp{\Gamma_2^{-1}z}{z} \Big)
	   -\dfrac{1}{2} \inp{\Gamma_2^{-1}z}{\inp{S_u}{\Gamma_2^{-1}S_u}S_u-(d+2)k_nS_u}
 	\\
 	\nonumber &  + \frac{1}{8}  \Big(  \inp{S_u}{\Gamma_2^{-1}S_u}^2- (4+2d) k_n  \inp{S_u}{\Gamma_2^{-1}S_u} +d(d+2) (k_n^2+k_n)- \mathrm{tr} (\Gamma_4\Gamma_2^{-2}) k_n  \Big) \bigg]  \Bigg\rbrace 
 \\ \label{eq-414} &
 +\alpha_{n,u} \frac{1}{n^{ 2+d/2}}  
	\end{align}
	where $\alpha_{n,u}(u\in\T_{k_n})$ denotes a family of infinitesimals dominated by an absolute infinitesimal $\epsilon_n$, i. e.
	\begin{equation*}
	\sup \{ |\alpha_{n,u}| : u\in\T_{k_n} \}\leq \epsilon_n \longrightarrow0.
	\end{equation*}
By  the definitions of  the martingales $N_{2,n}$,  we  deduce
\begin{align*}
&\frac{1}{m^{k_n}}\sum_{u\in \T_{k_n} }\Big(\inp{S_u}{\Gamma_2^{-1}S_u}    -dk_n \Big) = \inp{N_{2,k_n}}{\Gamma_2^{-1} \mathbf{1}},
\\
&\frac{1}{m^{k_n}}\sum_{u\in \T_{k_n} }	\bigg(  \inp{S_u}{\Big( \Lambda_d  - \frac{1}{4}  \inp{z}{\Gamma_2^{-1}z}\Gamma_2^{-1}\Big)  S_u} - k_n \Big(\tr{ \Gamma_2 \Lambda_d }- \frac{1}{4} d\inp{z}{\Gamma_2^{-1}z}\Big)\bigg)
\\ 
&\qquad\qquad\qquad\qquad= \inp{N_{2,k_n}}{\Big( \Lambda_d  - \frac{1}{4}  \inp{z}{\Gamma_2^{-1}z}\Gamma_2^{-1}\Big) \mathbf{1}}.
\end{align*}
By the linearity  of inner product and the definitions of $N_{1,n}, N_{3,n}$, we see  
	\begin{align*}
	& \frac{1}{m^{k_n}}  \sum_{u\in \T_{k_n} }\inp{S_u}{ \Gamma_2^{-1}z} 
	= \inp{N_{1,k_n}}{\Gamma_2^{-1}z},
	\\
&\frac{1}{m^{k_n}}\sum_{u\in \T_{k_n} }	\inp{ S_u}{\Big( 2\Lambda_d   - \frac{1}{2} \inp{z}{\Gamma_2^{-1}z} \Gamma_2^{-1}\Big) z}= \inp{N_{1,k_n}} {\Big( 2\Lambda_d   - \frac{1}{2} \inp{z}{\Gamma_2^{-1}z}\Gamma_2^{-1}\Big) z },
	\\
	& \frac{1}{m^{k_n}}\sum_{u\in \T_{k_n} }	 \inp{\Gamma_2^{-1}z}{\inp{S_u}{\Gamma_2^{-1}S_u}S_u-(d+2)k_nS_u}
	=\inp{\Gamma_2^{-1}z}{N_{3,k_n}}.
	\end{align*}	
Using the definitions of $N_{2,n}^z, N_{4,n}$ and substituting all the above  into  \eqref{eq-414}, 	we derive that
	\begin{align}
	\nonumber	\mathbb{D}_{2,n}&= \frac{1}{m^{k_n}}\sum_{u\in \T_{k_n}}  \P\Big(\widetilde{S}_{n-k_n}=z-y\Big)\bigg|_{y=S_u}
		\\ \label{eq-D2na}&= \dfrac{ (2\pi   n)^{-d/2}}{\sqrt{ \det \Gamma_2}}   \bigg[  W_{k_n}+\frac{1}{n} F_{1,k_n}(z)   + \frac{1}{n^2}F_{2,k_n}(z) \bigg]
			+\frac{1}{n^{ 2+d/2}}\bigg( \frac{1}{m^{k_n}}\sum_{u\in \T_{k_n}} \alpha_{n,u} \bigg).  
	\end{align}
where
	\begin{align*}
F_{1,k_n}(z)	=&\Big(\tau_d  -\dfrac{1}{2}\inp{z}{\Gamma_2^{-1}z} \Big)W_{k_n}+ \inp{N_{1,k_n}}{ \Gamma_2^{-1}z}  
	- \frac{1}{2}\inp{ N_{2,k_n}}{\Gamma_2^{-1}\mathbf{1}},
	\\
	F_{2,k_n}(z)= &\Big(\dfrac{1}{8}\inp{z}{\Gamma^{-1}z}^2-    \inp{z}{\Lambda_d z}+ \chi_d \Big)W_{k_n} 
	+  \inp{N_{1,k_n}} {\Big( 2\Lambda_d   - \frac{1}{2} \inp{z}{\Gamma_2^{-1}z} \Gamma_2^{-1}\Big) z}  
	\\ \nonumber 	
	& - \inp{ N_{2,k_n}}{ \Big( \Lambda_d  - \frac{1}{4}  \inp{z}{\Gamma_2^{-1}z}\Gamma_2^{-1}\Big) \mathbf{1} } + \dfrac{1}{2}  N_{2,k_n}^z-\dfrac{1}{2}  \inp{ \Gamma_2^{-1}z}{N_{3,k_n}}+\frac{1}{8}N_{4,k_n}.
	\end{align*}
	By the choice of $\kappa$, we see that $\lambda \kappa>2$ and $(\lambda -3)\kappa>1$. Hence by    \eqref{ConvW} and   Theorem \ref{thmCRM}, 
	\begin{align}
	\label{dSRW4.9} & W_{k_n}-W = o(1/n^2), \quad   {N}_{q,k_n}-\mathcal{V}_q = o(1/n), \quad q=1,   2 ,
	\\ \label{dSRW4.9a}
	&{N}_{2,k_n}^z-\mathcal{V}_2^z = o(1), \qquad  {N}_{q,k_n}-\mathcal{V}_q = o(1), \quad q=3,   4.
	 \end{align}
	 Therefore   as $n$ tends to infinity,  
	 \begin{equation*}
	 F_{1,k_n}(z)-F_1(z)=o(\frac{1}{n}),\quad F_{2,k_n}(z)-F_2(z)=o(1), ~~\mbox{a.s.}, 
	 \end{equation*}
	 where $F_{1}(z)$ and $F_2(z)$ are defined by \eqref{f1} and \eqref{f2}.
	 
	 Observe that
	 \begin{equation*}
	 \left   | \frac{1}{m^{k_n}}\sum_{u\in\T_{k_n}} \alpha_{n,u}\right|\leq  \epsilon_n W_{k_n} \rightarrow 0.
	 \end{equation*}
Substituting these into \eqref{eq-D2na}, we have 
\begin{equation}\label{eq-D2n}
	\mathbb{D}_{2,n}= \dfrac{ (2\pi   n)^{-d/2}}{\sqrt{ \det \Gamma_2}}   \bigg[  W_{k_n}+\frac{1}{n} F_{1}(z)   + \frac{1}{n^2}F_{2}(z) \bigg]
+\frac{1}{n^{ 2+d/2}}o(1),  
\end{equation}
with $F_{1}(z)$ and $F_2(z)$   defined by \eqref{f1} and \eqref{f2}. This is exactly what we want to prove.
\\
(II)	We consider the case $L\in \calB_d$.

 By using Theorem  \ref{Th2} (II) and following the  arguments as in  the first part (I),   the formula \eqref{SBRW-eq6a}
can be handled  and we omit the details.
	
	The lemma is proved. 
\end{proof}

 \section*{Acknowledgments}
 The author would like to thank the anonymous referees for valuable comments and suggestions  which significantly contributed to improving the quality of the publication.
 

\bibliographystyle{amsplain}

\begin{thebibliography}{10}
	
	\bibitem{ABMY2000}
	S.~Albeverio, L.~V. Bogachev, S.~A. Molchanov, and E.~B. Yarovaya,
	\emph{Annealed moment {L}yapunov exponents for a branching random walk in a
		homogeneous random branching environment}, Markov Process. Related Fields
	\textbf{6} (2000), no.~4, 473--516. 
	
	\bibitem{Asmussen1976AOP}
	S.~Asmussen, \emph{Convergence rates for branching processes}, Ann. Probab.
	\textbf{4} (1976), no.~1, 139--146. 
	
	\bibitem{AsmussenHering1983}
	S.~Asmussen and H.~Hering, \emph{Branching processes}, Progress in Probability
	and Statistics, vol.~3, Birkh\"auser Boston, Inc., Boston, MA, 1983.
	
	
	\bibitem{AsmussenKaplan76BRW1}
	S.~Asmussen and N.~Kaplan, \emph{Branching random walks. {I}}, Stochastic
	Process. Appl. \textbf{4} (1976), no.~1, 1--13.
	
	\bibitem{AthreyaNey72}
	K.~B. Athreya and P.~E. Ney, \emph{Branching processes}, Springer-Verlag, New
	York, 1972, Die Grundlehren der mathematischen Wissenschaften, Band 196.
	
	\bibitem{Biggins79}
	J.~D. Biggins, \emph{Growth rates in the branching random walk}, Z. Wahrsch.
	verw. Geb. \textbf{48} (1979), no.~1, 17--34.
	
	\bibitem{Biggins90SPA}
		J.~D. Biggins,\emph{The central limit theorem for the supercritical branching random
		walk, and related results}, Stochastic Process. Appl. \textbf{34} (1990),
	no.~2, 255--274.
	
	\bibitem{BinghamDoney1974AAP}
	N.~H. Bingham and R.~A. Doney, \emph{Asymptotic properties of supercritical
		branching processes. {I}. {T}he {G}alton-{W}atson process}, Advances in Appl.
	Probab. \textbf{6} (1974), 711--731.  
	
	\bibitem{Chen2001}
	X.~Chen, \emph{Exact convergence rates for the distribution of particles in
		branching random walks}, Ann. Appl. Probab. \textbf{11} (2001), no.~4,
	1242--1262.  
	
	\bibitem{ChenHe2017}
	X.~Chen and H.~He, \emph{On large deviation probabilities for empirical
		distribution of supercritical branching random walks with unbounded
		displacements}, Probab. Theory Relat. Fields \textbf{175} (2019), 255–307.
	
	\bibitem{Gao2016}
	Z.-Q. Gao, \emph{Exact convergence rate of the local limit theorem for
		branching random walks on the integer lattice}, Stoch. Process. Appl.
	\textbf{127} (2017), no.~4, 1282 -- 1296.
	
	\bibitem{Gao2018SPA}
	Z.-Q. Gao, \emph{A second order asymptotic expansion in the local limit theorem
		for a simple branching random walk in $\mathbb{Z}^d$}, Stoch. Process. Appl.
	\textbf{128} (2018), no.~12, 4000--4017.
	
	\bibitem{GL14}
	Z.-Q. Gao and Q.~Liu, \emph{Exact convergence rate in the central limit theorem
		for a branching random walk with a random environment in time}, Stoch.
	Process. Appl. \textbf{126} (2016), no.~9, 2634--2664.
	
	\bibitem{GL15}
		Z.-Q. Gao and Q.~Liu, \emph{Second and third orders asymptotic expansions for the
		distribution of particles in a branching random walk with a random
		environment in time}, Bernoulli \textbf{24} (2018), no.~1, 772--800.
	
	\bibitem{GLW14}
	Z.-Q. Gao, Q.~Liu, and H.~Wang, \emph{Central limit theorems for a branching
		random walk with a random environment in time}, Acta Math. Sci. Ser. B Engl.
	Ed. \textbf{34} (2014), no.~2, 501--512.  
	
	\bibitem{GrubelKabluchko2016AAP}
	R.~Grübel and Z.~Kabluchko, \emph{A functional central limit theorem for
		branching random walks, almost sure weak convergence and applications to
		random trees}, Ann. Appl. Probab. \textbf{26} (2016), no.~6, 3659--3698.
	
	\bibitem{GK15}
	R.~Gr\"{u}bel and Z.~Kabluchko, \emph{Edgeworth expansions for profiles of
		lattice branching random walks}, Ann. Inst. H. Poincaré Probab. Statist.
	\textbf{53} (2017), no.~4, 2103--2134.
	
	\bibitem{GunKonigS2013EJP}
	O.~G\"{u}n, W.~K\"{o}nig, and O.~Sekulovi\'{c}, \emph{Moment asymptotics for
		branching random walks in random environment}, Electron. J. Probab.
	\textbf{18} (2013), no. 63, 18.  
	
	\bibitem{Harris63BP}
	T.~E. Harris, \emph{The theory of branching processes}, Die Grundlehren der
	Mathematischen Wissenschaften, Bd. 119, Springer-Verlag, Berlin, 1963.
	
	\bibitem{HuangLiangLiu14}
	C.~Huang, X.~Liang, and Q.~Liu, \emph{Branching random walks with random
		environments in time}, Frontiers of Mathematics in China \textbf{9} (2014),
	no.~4, 835--842.
	
	\bibitem{HuangWangWang2020}
	C.~Huang, X.~Wang, and X.~Wang, \emph{Large and moderate deviations for a
		$\mathbb{R}^d$-valued branching random walk with a random environment in
		time}, Stochastics \textbf{92} (2020), no.~6, 944--968.
	
	\bibitem{iksanov_kabluchko_2016}
	A.~Iksanov and Z.~Kabluchko, \emph{A central limit theorem and a law of the
		iterated logarithm for the biggins martingale of the supercritical branching
		random walk}, J. Appl. Probab. \textbf{53} (2016), no.~4, 1178–1192.
	
	\bibitem{ILiuLiang2019}
	A.~Iksanov, X.~Liang, and Q.~Liu, \emph{On {$L^p$}-convergence of the {B}iggins
		martingale with complex parameter}, J. Math. Anal. Appl. \textbf{479} (2019),
	no.~2, 1653--1669.  
	
	\bibitem{Jagers74JAP}
	P.~Jagers, \emph{Galton-{W}atson processes in varying environments}, J. Appl.
	Probability \textbf{11} (1974), 174--178.
	
	\bibitem{JoffeMoncayo73AM}
	A.~Joffe and A.~R. Moncayo, \emph{Random variables, trees, and branching random
		walks}, Advances in Math. \textbf{10} (1973), 401--416. 
	
	\bibitem{Kabluchko12}
	Z.~Kabluchko, \emph{Distribution of levels in high-dimensional random
		landscapes}, Ann. Appl. Probab. \textbf{22} (2012), no.~1, 337--362.
 
	
	\bibitem{AsmussenKaplan76BRW2}
	N.~Kaplan and S.~Asmussen, \emph{Branching random walks. {II}}, Stochastic
	Process. Appl. \textbf{4} (1976), no.~1, 15--31.
	
	\bibitem{Lawler96}
	G.~F. Lawler, \emph{Intersections of random walks}, Modern Birkh\"auser
	Classics, Birkh\"auser/Springer, New York, 2013, Reprint of the 1996 edition.
 
	
	\bibitem{Lawler2010}
	G.~F. Lawler and V.~Limic, \emph{Random walk: a modern introduction}, Cambridge
	Studies in Advanced Mathematics, vol. 123, Cambridge University Press,
	Cambridge, 2010. 
	
	\bibitem{LiangLiu2020}
	X.~Liang and Q.~Liu, \emph{Regular variation of fixed points of the smoothing
		transform}, Stochastic Processes and their Applications \textbf{130} (2020),
	no.~7, 4104 -- 4140.
	
	\bibitem{LouidorPerkins2015EJP}
	O.~Louidor and W.~Perkins, \emph{Large deviations for the empirical
		distribution in the branching random walk}, Electron. J. Probab. \textbf{20}
	(2015), no. 18, 19.  
	
	\bibitem{Nakashima11}
	M.~Nakashima, \emph{Almost sure central limit theorem for branching random
		walks in random environment}, Ann. Appl. Probab. \textbf{21} (2011), no.~1,
	351--373.  
	
	\bibitem{Revesz94}
	P.~R{\'e}v{\'e}sz, \emph{Random walks of infinitely many particles}, World
	Scientific Publishing Co. Inc., River Edge, NJ, 1994.
	
	\bibitem{Shi2015}
	Z.~Shi, \emph{Branching random walks}, Lecture Notes in Mathematics, vol. 2151,
	Springer, Cham, 2015, Lecture notes from {\'E}cole d'{\'E}t{\'e} de
	Probabilit{\'e}s de Saint-Flour XLII -- 2012.  
	
	\bibitem{Stam66}
	A.~J. Stam, \emph{On a conjecture by {H}arris}, Z. Wahrscheinlichkeitstheorie
	und Verw. Gebiete \textbf{5} (1966), 202--206.
	
	\bibitem{WangHuang2017JTP}
	X.~Wang and C.~Huang, \emph{Convergence of martingale and moderate deviations
		for a branching random walk with a random environment in time}, J. Theoret.
	Probab. \textbf{30} (2017), no.~3, 961--995.  
	
	\bibitem{WangHuang2019ECP}
X.~Wang and C.~Huang,  \emph{Convergence of complex martingale for a branching random walk in
		a time random environment}, Electron. Commun. Probab. \textbf{24} (2019),
	Paper No. 41, 14.  
	
	\bibitem{WangLiuLiuLi2019}
	Y.~Wang, Z.~Liu, Q.~Liu, and Y.~Li, \emph{Asymptotic {P}roperties of a
		{B}ranching {R}andom {W}alk with a {R}andom {E}nvironment in {T}ime}, Acta
	Math. Sci. Ser. B (Engl. Ed.) \textbf{39} (2019), no.~5, 1345--1362.

	
	\bibitem{Yoshida08}
	N.~Yoshida, \emph{Central limit theorem for branching random walks in random
		environment}, Ann. Appl. Probab. \textbf{18} (2008), no.~4, 1619--1635.
 
	
	\bibitem{Zeitouni2012}
	O.~Zeitouni, \emph{Branching random walks and {G}aussian fields}, Notes for
	Lectures, http:// www. wisdom. weizmann.ac.il/$\sim$zeitouni/ pdf/
	notesBRW.pdf, 2012.
	
\end{thebibliography}

\end{document}